\documentclass[11pt]{amsart}
\usepackage{amsbsy,amssymb,amscd,amsfonts,latexsym,amstext,delarray,
amsmath,graphicx,color,caption}

\usepackage[percent]{overpic} 
\usepackage{tikz}
\usepackage{subcaption}

\usepackage{hyperref}
\hypersetup{
  colorlinks=true,
  linktoc=all,    
  linkcolor=black,
}

\usepackage{aliascnt} 
\usepackage{xparse}
\usepackage[capitalise,nameinlink]{cleveref} 

\let\oldtheorem\newtheorem
\RenewDocumentCommand{\newtheorem}{s m o m O{}}{%
\IfBooleanTF{#1}%
{\oldtheorem{#2}{#4}}%
{\IfNoValueTF{#3}{\oldtheorem{#2}{#4}[#5]}%
{\newaliascnt{#2}{#3}%
\oldtheorem{#2}[#2]{#4}%
\aliascntresetthe{#2}}}}

\usepackage{epsfig}
\input xy

\xyoption{all}
\def\C{{\mathbb C}}

\def\R{\mathbb{R}}
\def\Z{\mathbb{Z}}

\def\a{\alpha}

\def\lam{\lambda}

\def\bi{{\bf i}}

\def\cA{\mathcal A}

\def\cC{\mathcal C}

\def\cF{\mathcal F}
\def\cG{\mathcal G}

\newtheorem{theorem}{Theorem}[section]
 
\newtheorem{prop}[theorem]{Proposition}

\newtheorem{lemma}[theorem]{Lemma}
\newtheorem{definition}[theorem]{Definition}
\newtheorem{corollary}[theorem]{Corollary}

\numberwithin{table}{section}
\numberwithin{figure}{section}

\theoremstyle{remark}
\newtheorem{remark}[theorem]{Remark}

\providecommand{\abs}[1]{\left\lvert#1\right\rvert}
\providecommand{\cyc}[1]{\left\langle#1\right\rangle}

\providecommand{\set}[1]{\left\lbrace#1\right\rbrace}

\DeclareMathOperator{\Perm}{Perm}

\DeclareMathOperator{\Isom}{Isom}

\DeclareMathOperator{\Aut}{Aut}

\DeclareMathOperator{\PSL}{PSL}

\DeclareMathOperator{\matrixO}{O}

\DeclareMathOperator{\rk}{rk}
\DeclareMathOperator{\supp}{supp}

\DeclareMathOperator{\cone}{cone}

\DeclareMathOperator{\height}{ht}

\newcommand{\Waff}[1][]{\ensuremath{W^{\mathrm{aff}}_{#1}}}

\newcommand\restr[2]{{
  \left.\kern-\nulldelimiterspace 
  #1 
  \vphantom{\big|} 
  \right|_{#2} 
  }}

\newcommand\suchthat{ \mathrel{}\middle| \mathrel{}}

\renewcommand{\bold}[1]{\medskip \noindent {\bf #1 }\nopagebreak}

\begin{document}

\title{\bf{Piecewise isometry groups of Euclidean tessellations}}

\author{Robert Bieri, Alex Feingold, and Daniel Studenmund}

\date{30 July 2026}

\address{Department of Mathematics and Statistics, The State University of New York, Binghamton, New York 13902-6000}
\email{rbieri@binghamton.edu}
\email{feingold@binghamton.edu}
\email{dstudenm@binghamton.edu}

\begin{abstract}
  Given a tessellation of Euclidean or hyperbolic space, the piecewise
  isometry group is the group whose elements are given by cutting
  space into finitely many tessellated convex subsets and gluing them
  back together. Groups of piecewise isometries of tessellations
  generalize Houghton's groups and Thompson's group $V$, and for
  cubical tessellations were studied by Bieri and Sach. We prove
  structure results about groups of piecewise isometries of
  sufficiently nice tessellations of Euclidean space, such as
  tessellations associated to crystallographic root systems, in
  particular proving that they are elementary amenable. Future work in
  progress will prove finite generation and higher finiteness
  properties.
\end{abstract}

\maketitle


\section{Introduction}\label{intro}

Let $V$ be a finite-dimensional Euclidean or hyperbolic space, and
$\Omega$ a collection of codimension-1 totally geodesic subspaces in
$V$ such that $Isom(V,\Omega)$, the group of isometries of $V$
preserving $\Omega$, acts properly discontinuously on $V$ with finite
covolume and acts on $\Omega$ with finitely many orbits. Cutting $V$
along hyperplanes of $\Omega$ determines a tessellation $\Delta$ of
$V$. The piecewise isometry group $PI(\Delta)$ consists of
transformations $V\to V$ defined by cutting $V$ into finitely many
convex pieces along hyperplanes in $\Omega$ and pasting those pieces
by isometries bijectively onto the pieces of another such
decomposition of $V$.

The motivation of our project was triggered by the simple (but then
surprising) unifying observation of Bieri and Sach \cite{BS22}: The
piecewise isometric groups $PI(\Delta)$ constitute a geometrically
attractive and suggestive environment which hosts both Thompson's
groups and Houghton's groups. These two classes groups, which were
around but alien to one another for more than half a century, were now
prominently located in the low dimensional core of a new attractive
common homestead:
\begin{enumerate}
\item If $\Delta$ is the (unique) regular tessellation of the line,
  then the end-preserving subgroup of $PI(\Delta)$ quotiented by
  isometries of individual tiles is Houghton's group $H_2$. The higher
  Houghton group $H_n$ occur similarly when the tessellated line
  $\Delta$ is replaced by the disjoint union of $n$ parallel tessellated
  rays.
\item If $\Delta$ is the (unique) regular tessellation of the
  hyperbolic plane by ideal simplices, one observes that all pieces of
  finite decompositions are either half-spaces or finite
  tile-clusters, and the decompositions are partially ordered along a
  tree. Then the quotient of $PI(\Delta)$ by the subgroup of elements
  supported on finitely many triangles is Thompson's group $V$. This
  observation seems to be folklore; see Bieri--Sach
  \cite[\S2.6]{BS22}.
\end{enumerate}

Thompson's groups were the starting point of a whole new development
in group theory with an extended literature. This includes work based
on their faithful piecewise isometric action on the hyperbolic plane
with applications to (asymptotically rigid) mapping class groups of
surfaces \cite{Penner, FKS12, GLU22, ABKL24}.

In this paper, we consider the case that $V$ is a Euclidean space,
and $\Delta$ is cut out by finitely many families of parallel
hyperplanes such that the group of isometries preserving $\Delta$ acts
properly discontinuously and cocompactly. The most important examples
of tessellations satisfying these hypotheses are given by those
associated to crystallographic root systems \cite{Hum}, where the
isometry group of the tessellation is virtually an affine Weyl
group. Not all such tessellations arise directly from root systems:
for example, the trihexagonal tiling of the plane, also called the
kagome lattice. Our main theorem generalizes a result of Bieri and
Sach \cite{BS22}.

\begin{theorem} \label{thm:elem-amen} Let $\Delta$ be a tessellation
  of Euclidean space as above. Then $PI(\Delta)$ is elementary
  amenable.
\end{theorem}

Theorem \ref{thm:elem-amen} is proved by using the geometry of the
tessellation to define a normal series in $PI(\Delta)$ and analyzing
the successive quotients. Given a convex polyhedral set $C$ cut out by
hyperplanes in $\Omega$, define its limit set $L(C)$ to be the set of
unit vectors $u\in V$ such that $C$ contains a ray in the direction
$u$. Define the rank of $C$ to be one more than the dimension of its
limit set, and say that $C$ is irreducible if it cannot be decomposed
as a union of two convex polyhedral sets of equal rank with disjoint
interiors. Define a {\em germ} to be an equivalence class of
irreducible convex polyhedral sets, where two are equivalent if each
has the same rank as the interior of their intersection.

The piecewise isometry group $PI(\Delta)$ acts on the set of germs
$\Gamma$ preserving rank. We use this action to define normal
subgroups $C_r^+, C_r, G_r \leq PI(\Delta)$ for each
$r=0,\dotsc, \dim(V)$, where $C_r$ is subgroup fixing each germ of
rank $r$. The main theorem follows immediately from the following
summary of the results of \cref{sec:sequences}.

\begin{prop} \label{prop:introchain}
  There is a chain of normal subgroups of $PI(\Delta)$
  \[
    \{1\} = C_0^+ \leq C_0 \leq G_0 \leq C_1^+ \leq \dotsb \leq
    C_{\dim(V)} \leq G_{\dim(V)} = PI(\Delta)
  \]
  where each quotient $G_r/C_r$ or $C_r/C_r^+$ is locally finite and
  each quotient $C_r^+/G_{r-1}$ is abelian.
\end{prop}

Future work will build on this description to prove finiteness
properties of the group $PI(\Delta)$, further generalizing results of
Bieri and Sach.

\subsection{Examples}

Our intuition about piecewise isometries was built by studying
examples arising from crystallographic root systems. Such root systems
are products of irreducible root systems, which are completely
classified. The irreducible root system of type $A_1$ gives rise to a
piecewise isometry group of the real line whose quotient by the
locally finite normal subgroup $C_0$ (as defined above) is Houghton's
group $H_2$. The reducible root system of type $(A_1)^n$ gives rise to
the piecewise isometry group studied by Bieri and Sach \cite{BS22}.

In \cref{sec:examples}, we give an explicit description of the set
of germs for crystallographic tessellations. We include illustrations
for the root system of type $A_2$, which generates the tessellation of
the plane by equilateral triangles, and provide examples of piecewise
isometries lying at different terms in the chain of normal subgroups
in \cref{prop:introchain}. We also provide pictures of irreducible
convex polyhedral sets for the tessellation corresponding to the root
system of type $A_3$ in $\R^3$. As examples of tessellations which are
not coming from an affine Weyl group we discuss the 2D and 3D kagome
tessellations. We include many figures to illustrate the beautiful
geometry in these examples, and hope that they spark joy for the reader
and they did for us.

\subsection{Context}

Piecewise-defined maps have long been a source of interesting examples
of groups. For example, the regular tessellation of the hyperbolic
plane by ideal triangles has orientation-preserving isometry group
$\PSL_2(\Z)$. The subgroup of $\operatorname{Homeo}(S^1)$ composed of
elements which are piecewise-$\PSL_2(\Z)$ is isomorphic to Thompson's
group $T$, a result attributed to Thurston \cite{Greenberg90}. Imbert
shows moreover that both are isomorphic to Penner's universal Ptolemy
group \cite{Imbert97,Penner}, a group generated by edge flips of ideal
triangle tessellations of the hyperbolic plane equipped with a
distinguished oriented edge. An application of that example in physics
was found by Osborne and Stiegemann \cite{OS_2020}.

One way to generalize this is to realize $PSL(2,\Z)$ as the even
subgroup of the hyperbolic Weyl group,
$$PGL(2,\Z)\cong T(2,3,\infty)$$
of a rank $3$ hyperbolic Kac-Moody Lie algebra,
$\cF = AE3 = A_1^{++}$, studied by Feingold and Frenkel \cite{FeFr83}
in 1983. Weyl groups of two rank $4$ hyperbolic Kac-Moody Lie algebras
contain finite-index subgroups $PSL(2,E)$ and $PSL(2,\Z[\bi])$, where
$E = \Z[e^{\pi\bi/3}]$ is the ring of Eisenstein integers and
$\Z[\bi]$ is the ring of Gaussian integers.  These Weyl groups were
studied by Feingold-Kleinschmidt-Nicolai
\cite{FeKlNi09,FeKlNi09a,FeKlNi09b} in 2009 in relation to normed
division algebras, and by Feingold-Valli\`eres \cite{FV18} in 2017 in
relation to Clifford algebras.  They are arithmetic subgroups of
$PSL(2,\C)$ which act properly discontinuously on hyperbolic
$3$-space, $H^3$. One motivation of our current project is to study
piecewise isometry groups of the tessellations of $H^3$ determined by
these hyperbolic Weyl groups, giving us definitions of new groups
which could be called $PPSL(2,E)$ and $PPSL(2,\Z[\bi])$. Associated to
each ideal vertex of these hyperbolic tessellations is a Euclidean
tessellation of a sufficiently small horosphere, and the piecewise
isometry groups of the Euclidean $A_2$ and $B_2$ tessellations arise
as parabolic subgroups of the hyperbolic piecewise isometry group.

Cutting and pasting also appears in the context of Hilbert's third
problem, which asks whether, given two polyhedra of equal volume, one
can be cut into finitely many polyhedral pieces and reassembled to
form the second. While the question is long-settled, research into
cut-and-paste equivalence continues using the Zakharevich's language
of assemblers \cite{Zak17}. See \cref{rem:zakh} for a definition of our
piecewise isometry groups following Kupers et al.~\cite{Kup24}.


\section{Piecewise isometries} \label{sec:piecewiseisometries}

Group actions will be right actions, following much of the literature
of piecewise groups. When it will not cause confusion, elements of
groups that are naturally functions may be written on the left in
function notation.

\subsection{Tessellations} \label{sec:tessellations}

Let $(V, (\cdot, \cdot))$ be a finite-dimensional real inner product
space. Let $\Omega$ be a collection of affine hyperplanes in $V$ such
that $\Isom(V,\Omega)$, the group of isometries of $V$ preserving
$\Omega$, acts on $V$ properly discontinuously and cocompactly, and
acts on $\Omega$ with finitely many orbits. 
This determines a tessellation $\Delta$ of $V$. The open top-dimensional tiles are
connected components of $V - \bigcup_{H\in \Omega} H$, called regions
of the arrangement, and their facets are regions of intersection
subspaces. Then $\Delta$ is a locally finite polyhedral complex with
isometry group $\Isom(\Delta)$ acting cocompactly on
$V$. Tessellations formed in this way are the central object of our
study.

By assumption, $\Isom(\Delta)$ is a lattice in the Lie group
$\Isom(V)$, and so the translation subgroup $T\leq \Isom(\Delta)$ has
finite index. It follows then that $\Omega$ admits a decomposition as
a disjoint union $\Omega = \Omega_1 \sqcup \dotsb \sqcup \Omega_m$ of
finitely many families of parallel hyperplanes. Given a hyperplane
$H\in \Omega$, let $\Omega_H \subset \Omega$ denote its parallelism
class.

Let $\mathcal{H}$ be the collection of all closed half-spaces
$M \subset V$ bounded by hyperplanes $H\in \Omega$. Let
$\iota:\mathcal{H} \to \mathcal{H}$ be the involution sending each
half-space $M$ to the half-space $\iota(M) = \overline{V-M}$. The map
$\mathcal{H}\to \Omega$ mapping $M\mapsto \partial M$ is a 2-to-1
surjection.

By a {\em maximal nested family} in $\mathcal{H}$ we mean a maximal
subset $\xi \subset \mathcal{H}$ with the property that for any two
half-spaces $M,M'\in \xi$, either $M\subset M'$ or $M'\subset
M$. Each maximal nested family in $\mathcal{H}$ is of the form
\[
  \xi_M = \set{ M'\in \mathcal{H} \suchthat M\subset M' \text{ or
    }M'\subset M}
\]
for some $M\in \mathcal{H}$. Each maximal nested family is totally
ordered by containment, and the boundaries of its elements constitute
a parallelism class $\Omega_H$ for some $H\in \Omega$.

\begin{definition} \label{def:phi}
Let $\Xi$ be the collection of all maximal nested families in
$\mathcal{H}$. The involution $\iota: \mathcal{H} \to \mathcal{H}$
descends to an involution of $\Xi$, which by abuse of notation we
also denote by $\iota$.
\end{definition}

\subsection{Convex polyhedral sets}

The building blocks for piecewise isometries are convex polyhedral
sets. 

\begin{definition}
  A {\em convex polyhedral subset} of $V$ is a subset $P\subseteq V$
  that is either empty or is a set with nonempty interior equal to the
  intersection finitely many closed half-spaces bounded by hyperplanes
  $H\in \Omega$.
\end{definition}

The full space $V$ and the empty set $\emptyset$ are convex polyhedral
subsets. Note that a nonempty intersection of such half-spaces has
empty interior if and only if it is contained in some hyperplane
$H\in \Omega$. The intersection of two convex polyhedral sets may have
empty interior, and therefore not be convex polyhedral. To remedy
this, define the {\em essential intersection} of convex polyhedral
sets $P,Q$ as
\[
  P \doublecap  Q = \begin{cases} \emptyset & \text{ if $P\cap Q$ has
      empty interior}\\
    P\cap Q & \text{ else.} \end{cases}
\]
Say that $P$ and $Q$ are {\em essentially disjoint} if $P\doublecap Q
= \emptyset$.

\begin{definition}
  A {\em finite decomposition} of a subset $X\subseteq V$ is a finite
  collection of convex polyhedral subsets $P_1,\dotsc, P_m\subseteq X$
  whose union is $X$ and which are pairwise essentially disjoint. If
  $X\subseteq V$ admits a finite decomposition, call it a {\em
    polyhedral set}.
\end{definition}

Note that the collection of polyhedral sets is closed under
intersections, unions, and complements.

\subsection{The limit complex at infinity}

A {\em direction} in $V$ is a parallelism class of rays in $V$. Let
$\partial V$ be the set of directions in $V$, thought of as the sphere
at infinity. Each convex subset $C\subset V$ has a limit set $L(C)$,
the set of directions contained in $C$. For any hyperplane
$H\in \Omega$, its limit set $L(H) \subset \partial V$ is a great
subsphere. The collection of all sets $L(H)$ for $H\in\Omega$ is a
finite collection of great subspheres, one for each parallelism class
$\Omega_H$.

The collection of all $L(H)$ for $H\in \Omega$ gives $\partial V$ the
structure of a spherical cell complex. The open cells of top dimension
are the components of the deleted space
$\partial V - \bigcup_{H\in \Omega} L(H)$. The open cells of dimension
$k$ are top-dimensional cells of the subsphere arrangements induced in
$k$-dimensional subspheres $S^k \subset \partial V$ which are
intersections of limit sets $L(H)$. In particular, the vertices of the
complex are elements $v\in \partial V$ such that the antipodal pair
$\{v, -v\}$ is equal to the intersection of a finite collection of
sets $L(H)$ for $H\in \Omega$. The collection of all open cells,
including the $(-1)$-dimensional empty cell, will be written
$L(\Delta)$. This complex has the structure of the set of (co)vectors
of an oriented matroid \cite[\S1.3.4]{Anderson}.

Each maximal nested family $\xi\in \Xi$ has a well-defined limit
set $L(\xi)$, by definition equal to the closed hemisphere $L(M)$
for any $M\in \xi$. The intersection
$L(\xi) \cap L(\iota(\xi))$ is equal the boundary
$\partial L(\xi)$, which is equal to the great subsphere
$L(\partial M)$ where $\partial M$ is the boundary hyperplane of any
representative $M\in \xi$. The closure of any cell in $L(\Delta)$
is then equal to an intersection of hemispheres
$L(\xi_1),\dotsc, L(\xi_n)$ for some
$\xi_1,\dotsc,\xi_n \in \Xi$, and such an intersection
represents the closure of an open cell if and only if it is minimal
among all such intersections of its dimension.

If convex sets $C$ and $D$ have nonempty intersection, then
$L(C\cap D) = L(C)\cap L(D)$. In particular, if a nonempty convex
polyhedral set $P$ is the intersection of half-spaces
$M_1,\dotsc, M_n$, then $L(P)$ is the intersection of the closed
hemispheres $L(M_1),\dotsc, L(M_n)$. Therefore $L(P)$ has the
structure of a subcomplex of the spherical cell complex $L(\Delta)$.

Let $S(V)$ be the unit sphere in $V$. There is an identification
between $\partial V$ and $S(V)$ sending each parallelism class of rays
to the unit vector in its direction. Under this identification, for
any convex set $C\subset V$ the limit set $L(C) \subset \partial V$ is
identified with $\cone(C)\cap S(V)$, where $\cone(C)$ is the recession
cone of $C$. The {\em recession cone} of a convex set $C$ is the set
of vectors $v\in V$ such that $x+tv\in C$ for any $t>0$ and $x\in C$
(cf.~\cite[1.5]{Zieg}). The recession cone is a cone at the origin.

\subsection{Definition of piecewise isometries}

We now define the main object of our study, the group $G = PI(\Delta)$
of piecewise isometries of $V$ determined by $\Delta$. For the
remainder of this section, let $\Gamma$ denote a subgroup of
$\Isom(\Delta)$. When $\Gamma$ is omitted from notation we assume
$\Gamma = \Isom(\Delta)$.

\begin{definition}
  A {\em piecewise-$\Gamma$ embedding} is an injective function
  $f:X_1 \to X_2$ between polyhedral subsets $X_1, X_2 \subset V$ such
  that $X_1$ admits a finite decomposition by convex polyhedral sets
  $P_1,\dotsc, P_n$ such that the restriction of $f$ to the interior
  each $P_i$ extends to an element of $\Gamma$.
\end{definition}

If the restriction of $f$ to the interior of a convex polyhedral set
$P$ extends to an element of $\Gamma$, we say that $f$ is
{\em $\Gamma$-isometric on $P$}. When $\Gamma$ is the full isometry
group $\Isom(\Delta)$, say simply that $f$ is isometric on $P$.

\begin{lemma}\label{lem:piecebasics}
  Consider subsets $X_1,X_2 \subset V$, a function
  $f : X_1 \to X_2$, and a convex polyhedral subset $P \subset X_1$ on
  which $f$ is $\Gamma$-isometric.
  \begin{enumerate}
  \item $f(P) \subset V$ is a convex polyhedral subset.
  \item Given a finite decomposition $Q_1,\dotsc, Q_m \subset f(P)$ of
    $f(P)$, setting \mbox{$P_i = \restr{f}{P}^{-1}(Q_i)$} defines a
    finite decomposition of $P$ by convex polyhedral subsets
    $P_1,\dotsc, P_m \subset P$ on which $f$ is $\Gamma$-isometric.
  \end{enumerate}
\end{lemma}
\begin{proof}
  Evident.
\end{proof}

\begin{lemma} \label{lem:compose}
  The composition of two piecewise-$\Gamma$ embeddings is a
  piecewise-$\Gamma$ embedding.
\end{lemma}
\begin{proof}
  Given two piecewise-$\Gamma$ embeddings $f:X_1 \to X_2$ and
  $g:X_2 \to X_3$, let $P_1,\dotsc, P_m \subset X_1$ and
  $Q_1,\dotsc, Q_k \subset X_2$ be finite decompositions of $X_1$ and
  $X_2$ by convex polyhedral sets on which $f$ and $g$ are,
  respectively, $\Gamma$-isometric. Then for each $i$, the collection
  $Q_1 \doublecap f(P_i), \dotsc, Q_k \doublecap f(P_i)$ is a finite
  decomposition of $f(P_i)$, where $\doublecap$ is the essential
  intersection. Then by \cref{lem:piecebasics}, the collection
  $f^{-1}\left(Q_j \doublecap f(P_i) \right)$ forms a finite
  decomposition of $X_1$ by convex polyhedral sets on which $g\circ f$
  is $\Gamma$-isometric.
\end{proof}

Because the image of the interior of a convex polyhedral set under a
$\Gamma$-isometry is again the interior of a convex
polyhedral subset, the inverse of an invertible piecewise-$\Gamma$
embedding is again a piecewise-$\Gamma$ embedding. Say that an
invertible piecewise-$\Gamma$ embedding is a {\em piecewise-$\Gamma$
  isometry}.

\begin{definition}
  The {\em piecewise-$\Gamma$ isometry group of $\Delta$} is the group
  $PI(\Delta,\Gamma)$ of all piecewise-$\Gamma$ isometries $f:V \to V$
  under composition. When $\Gamma = \Isom(\Delta)$, we call this the
  {\em piecewise isometry group of $\Delta$} and write $PI(\Delta)$.
\end{definition}

\begin{remark}\label{rem:zakh}
  Piecewise isometries of tessellations may also be defined using
Zakharevich's language of assemblers \cite{Zak17}, following the
presentation of Kupers et al. \cite{Kup24}. We outline the
construction here, though we will not use it in the sequel. Let
$\cA_\Gamma$ be the category whose objects are convex polyhedral sets
and whose morphisms $P\to Q$ for nonempty $P$ are $f\in \Gamma$ such
that $f(P) \subset Q$, and for $P=\emptyset$ there is a unique
morphism $P\to Q$ for each $Q$. A covering family is a collection of
morphisms with the same target, the union of whose images is equal to
the target. This gives $\cA_\Gamma$ the structure of an
assembler. Then, in the notation of Kupers at al. \cite{Kup24}, the
piecewise isometry group $PI(\Delta, \Gamma)$ is the group
$\Aut_{\cG(\cA_\Gamma)}(V)$. Note however that $\cA$ is neither an
EA-assembler nor an S-assembler.
\end{remark}

\subsection{Translations} Recall that $T\leq \Isom(\Delta)$, the
group of all translations preserving the tessellation, acts properly
discontinuously and cocompactly on $V$. Given any hyperplane $H\in
\Omega$, there is a positive lower bound on the distance between $H$
and $Ht$ for $t\in T$, and therefore there is a translation $t_H\in T$
minimizing the distance.

\begin{lemma} \label{lem:stab-decomp}
  Given any hyperplane $H\in \Omega$, let $stab_T(H)$ be the
  stabilizer of $H$ in $T$. For any $t_H\in T$ minimizing the nonzero
  distance between $H$ and $Ht_H$, there is a direct sum decomposition
  $T \cong stab_T(H) \oplus \cyc{t_H}$.
\end{lemma}
\begin{proof}
  The quotient group $T / stab_T(H)$ acts on the set of hyperplanes
  $\Omega_H$. Given $\phi \in T/ stab_T(H)$, chose the $N$
  such that $H t_H^N$ lies closest to $H \phi$. Then $H \phi = H t_H^N$
  by minimality of $t_H$. The assignment $\phi \mapsto t^N$ is a
  section of the quotient map $T \to T/stab_T(H)$.
\end{proof}

\begin{lemma} \label{lem:trans-hyperplane}
  For each hyperplane $H\in \Omega$, the stabilizer $stab_T(H) \leq T$
  acts cocompactly on $H$. 
\end{lemma}
\begin{proof}
  Pick any $t_H\in T$ minimizing the nonzero distance between $H$ and
  $Ht_H$, and let $A \subset V$ be the convex set bounded by $H$ and
  $H t_H$. Let $p\in V$ be any point equidistant from $H$ and $H
  t_H$. A fundamental domain $D$ for the action of $T$ on $V$ is given
  by a Voronoi region about $p$, which is contained in $A$. By
  \cref{lem:stab-decomp}, the union of the $stab_T(H)$-orbit of $D$ is
  all of $A$. It follows that the action of $stab_T(H)$ on $A$ is
  cocompact, and so the action of $stab_T(H)$ on $H$ is cocompact.
\end{proof}

\begin{corollary} \label{cor:line-trans}
  If $L\subset V$ is an affine line that is the intersection of
  hyperplanes in $\Omega$, then $stab_T(L)$ acts cocompactly on $L$.
\end{corollary}
\begin{proof}
  Given any $H\in \Omega$, choosing coordinates on $H$ gives $H$ the
  structure of an inner product space with a distinguished collection
  of hyperplanes, the collection of all proper subsets $H\cap H'$ for
  $H'\in \Omega$. This defines a tessellation of $H$ which has a
  cocompact and properly discontinuous group of isometries by
  \cref{lem:trans-hyperplane} with finitely many orbits of
  hyperplanes. The result now follows from \cref{lem:trans-hyperplane}
  by induction.
\end{proof}

\section{Structure of convex polyhedral sets}

\subsection{Bounds and rank}

Each nonempty convex polyhedral subset $P$ is a finite intersection of
half-spaces in $\mathcal{H}$. Since each maximal nested family is
discretely linearly ordered, this intersection contains minimal
elements in each such family. Here we introduce notation for this.

For each $\xi \in \Xi$, say that $P$ is {\em $\xi$-bounded} if
$P\subset M$ for some $M\in \xi$. If no such bounding half-space
exists, say that $P$ is {\em $\xi$-unbounded}. Let
$\Xi_P \subset \Xi$ be the collection of $\xi$ for which $P$ is
$\xi$-bounded. Given $\xi \in \Xi_P$, let
$b_P(\xi) \in \mathcal{H}$ be the minimal half-space $M\in \xi$
such that $P \subset M$. We evidently have the following description
of $P$:
\begin{equation}\label{eq:piecebounds}
  P = \bigcap_{\xi \in \Xi_P} b_P(\xi).
\end{equation}

A key tool in understanding the structure of polyhedral sets will be
decompositions into irreducible pieces. The appropriate notion of
irreducibility is defined using the concept of rank.

\begin{definition}
  The {\em rank} of a convex subset $P \subset V$, denoted $\rk(P)$,
  is the maximal number of linearly independent rays in $P$. Define
  the {\em rank} of the empty set to be $-1$.
\end{definition}

It is clear that if $Q\subseteq P$ then $\rk(Q)\leq \rk(P)$.  In fact,
the rank of a convex polyhedral set is related to the dimension of its
limit set.

\begin{lemma}\label{lem:limitdim}
  If $P$ is a convex polyhedral set of rank $k$ then
  $L(P)\subset \partial V$ is a closed submanifold with corners of
  dimension $k-1$, equal to the intersection of finitely many closed
  hemispheres.
\end{lemma}
\begin{proof}
  From \cref{eq:piecebounds} we deduce the equality
  \begin{equation} \label{eq:limitsetdesc}
    L(P) = \bigcap_{\xi\in \Xi_P} L(\xi),
  \end{equation}
  hence $L(P) \subset \partial V$ is a closed submanifold with
  corners. For the computation of dimension, let $W \subset V$ be the
  linear span of the recession cone, $\cone(P)$. Then $L(P)$ is the
  intersection of the codimension-0 submanifold $\cone(P) \subset W$
  with the codimension-1 submanifold $S(V) \cap W \subset W$, so
  $L(P)$ has dimension one less than that of $W$. This completes the
  proof, since the dimension of $W$ is equal to the rank of $P$.
\end{proof}

\begin{lemma} \label{lem:unboundeddirection} Suppose $P$ is a convex
  polyhedral set. For each $\xi\in \Xi$, there is a containment
  $L(P)\subset L(\xi)$ if and only if $P$ is $\xi$-bounded.
\end{lemma}
\begin{proof}
  It is clear that if $P$ is $\xi$-bounded then $L(P)\subset
  L(\xi)$. Conversely, let $\Xi_P$ be the set of $\xi\in \Xi$ such
  that $P$ is $\xi$-bounded and consider any $\beta \in \Xi -
  \Xi_P$. For each $\xi\in \Xi$ choose a vector $v_\xi\in V$ that is a
  positive normal vector to the hyperplane $\partial M$ for any
  $M\in \xi$. Since $P$ is $\beta$-unbounded, clearly $v_\beta$ cannot
  be written as a nonnegative linear combination of $v_\xi$ for
  $\xi\in \Xi_P$.  It follows from a standard formulation of the
  Farkas Lemma (for example, see Farkas Lemma 3 of Anderson
  \cite[\S1.5.2, page 30]{Anderson}) that there is a unit vector
  $v\in S(V)$ such that $(v,v_\beta) < 0$ and $(v,v_\xi) \geq 0$ for
  all $\xi\in \Xi_P$. It follows from the latter conditions and
  Equation \cref{eq:limitsetdesc} that $v\in L(P)$, and from the former
  condition that $v\notin L(\xi)$.
\end{proof}

We are interested in how rank changes upon ``cutting'' a polyhedral
set by a hyperplane. That is, we are interested in relating the rank
of $P$ to the rank of $P\doublecap M$ for $M\in \mathcal{H}$. The relation
between the ranks of these subsets depends on the boundedness of $P$.

\begin{prop} \label{prop:piececut} Suppose $P$ is a nonempty
  convex polyhedral set.  Fix $\xi\in \Xi$ and a half-space
  $M\in \xi$ such that $P\doublecap M$ is nonempty.
  \begin{enumerate}
  \item \label{unboundcutpart} If $P$ is $\xi$-bounded, then  
    $\rk(P\doublecap M) = \rk(P)$.
  \item \label{boundcutpart} If $P$ is $\xi$-unbounded and 
    $\iota(\xi)$-bounded, then  $\rk(P\doublecap M) < \rk(P)$.
  \item \label{doubleunboundcutpart} If $P$ is unbounded with respect
    to both $\xi$ and $\iota(\xi)$, then $\rk(P\doublecap M) = \rk(P)$.
  \end{enumerate}
\end{prop}
\begin{proof}
  First suppose $P$ is $\xi$-bounded. Then \cref{unboundcutpart}
  follows from Equation \eqref{eq:limitsetdesc} since  
  \[ L(P\doublecap M) = L(P)\doublecap L(M) = \bigcap_{\xi \in
      \Xi_P} L(\xi)  \doublecap L(M) = L(P).
  \]
  
  Now suppose $P$ is $\xi$-unbounded and $\iota(\xi)$-bounded.  It
  follows from Equation \eqref{eq:limitsetdesc} that
  $L(P\doublecap M)$ is contained in $\partial L(\xi)$. By
  \cref{lem:unboundeddirection} there is some
  $v\in L(P) - \partial L(\xi)$. By spherical convexity of the limit
  set $L(P)$, there is a spherical segment from $v$ to $L(P\doublecap M)$
  intersecting $L(P\doublecap M)$ transversely, so the dimension of $L(P)$
  is at least one greater than the dimension of $L(P\doublecap M)$. This
  proves \cref{boundcutpart}.
   
  Now suppose $P$ is both $\xi$-unbounded and
  $\iota(\xi)$-unbounded.  By \cref{lem:unboundeddirection}, find
  $u \in L(P)- L(\xi)$ and $v\in L(P)- L(\iota(\xi))$. By
  convexity, the limit set $L(P)$ contains a spherical line segment
  from $u$ to $v$. This segment intersects the $(n-2)$-sphere
  $\partial L(\xi)$ transversely inside $L(P)$, so
  $L(P) \cap \partial L(\xi)$ has codimension 1 inside
  $L(P)$. Since $L(P\doublecap M)$ contains both
  $L(P) \cap\partial L(\xi)$ and $v$, it follows that the rank of
  $P\doublecap M$ is equal to $(\rk(P)-1)+1$, proving
  \cref{doubleunboundcutpart}.
\end{proof}

\subsection{Irreducibles}

The first step in defining a space on which $PI(\Delta)$ acts is to
define a set of classes of certain polyhedral sets on which
$PI(\Delta)$ acts. Because a piecewise isometry may not have a
well-defined associated isometry on a given convex polyhedral set, the
polyhedral sets we consider will be `small' in a sense that we define
in \cref{def:irred} and alternatively characterize in
\cref{lem:irred-form}.

\begin{definition} \label{def:irred} A convex polyhedral set $P$ is
  {\em irreducible} if, for any finite decomposition
  $P_1,\dotsc, P_m$ of $P$, exactly one of the sets
  $P_i$ has rank equal to the rank of $P$.
\end{definition}

It follows from Equation \eqref{eq:limitsetdesc} by induction that a
convex polyhedral set $P$ is irreducible if and only if for each
$M\in \mathcal{H}$, exactly one of $P\doublecap M$ and
$P\doublecap \iota(M)$ has rank equal to the rank of $P$. We will
often use this characterization of irreducibility in what follows.

\begin{definition} \label{def:thin}
  A nonempty convex polyhedral set $P$ is {\em thin} if for all
  $\xi \in \Xi$ with the property that $P$ is both
  $\xi$-bounded and $\iota(\xi)$-bounded, for any $M\in \xi$
  either $P\doublecap M = P$ or $P\doublecap M = \emptyset$.
\end{definition}

To understand this another way, define a {\em $\xi$-layer} to be a
nonempty polyhedral set of the form $M\doublecap M'$ some $M\in
\xi$ and $M'\in \iota(\xi)$. The set of $\xi$-layers is a
directed poset ordered by inclusion. Then a convex polyhedral set $P$
is thin if, for each $\xi \in \Xi$ such that $P$ is contained in
an $\xi$-layer, $P$ is contained in a minimal $\xi$-layer. Note
that an $\xi$-layer is thin if and only if it is minimal.

\begin{lemma} \label{lem:irred-form} For a nonempty convex polyhedral
  set $P$, the following are equivalent:
  \begin{enumerate}
  \item \label{cond-irred} $P$ is irreducible.
  \item \label{cond-bound} $P$ is thin and for each $\xi \in \Xi$, $P$
    is $\xi$-bounded or $\iota(\xi)$-bounded.
  \item $P$ is thin and $L(P)$ is a closed cell in the spherical
    complex $L(\Delta)$.
  \end{enumerate}
\end{lemma}

\begin{proof}
  To see that irreducibility implies the second condition, suppose $P$
  is irreducible. First suppose there exists $\xi \in \Xi$ such that
  $P$ is both $\xi$-unbounded and $\iota(\xi)$-unbounded. Then for any
  $M\in \xi$, both $P \cap M$ and $P\cap \iota(M)$ have rank equal to
  the rank of $P$ by \cref{prop:piececut}, while
  $P = (P\cap M) \cup (P\cap \iota(M))$ is a finite decomposition,
  contradicting irreducibility of $P$.
  
  Now suppose $P\subset M\cap M'$ where $M\in \xi$ and
  $M'\in \iota(\xi)$, and there is $M''\in\xi$ such that both
  $P\doublecap M''$ and $P\doublecap \iota(M'')$ are nonempty. Then
  both $P \cap M''$ and $P\cap \iota(M'')$ have rank equal to the rank
  of $P$ by \cref{prop:piececut}, while
  $P = (P\cap M'') \cup (P\cap \iota(M''))$ is a finite decomposition,
  contradicting irreducibility of $P$.
  
  To see that the second condition irreducibility, suppose $P$ is
  reducible, so that there is a finite decomposition
  $P = (P\cap M) \cup (P\cap \iota(M))$ for some $M\in \xi$, where
  both $P\cap M$ and $P\cap \iota(M)$ have the same rank as
  $P$. Suppose that $P$ is $\xi$-bounded or
  $\iota(\xi)$-bounded. By \cref{prop:piececut}, $P$ must be both
  $\xi$-bounded and $\iota(\xi)$-bounded. But since both
  $P\cap M$ and $P\cap \iota(M)$ are nonempty, $P$ cannot be thin.

  To finish the proof, we assume $P$ is thin and show that $P$ is
  $\xi$-bounded or $\iota(\xi)$-bounded for each $\xi\in \Xi$ if and
  only if its limit set is a closed cell in $L(\Delta)$.  Let $\Xi_P$
  be the set of $\xi\in \Xi$ for which $P$ is $\xi$-bounded. Consider
  $\Xi_P \cap \iota(\Xi_P)$, the set of $\xi$ for which $P$ is both
  $\xi$-bounded and $\iota(\xi)$-bounded. Consider the subsphere
  \[
    S = \bigcap_{\xi \in \Xi_P\cap \iota(\Xi_P)} L(\xi).
  \]
  Now, $L(P)$ is the intersection of $S$ with a closed half-space for
  each $\beta \in \Xi$ for which $P$ is $\beta$-bounded but not
  $\iota(\beta)$-bounded. The desired result now follows because the
  resulting set is a closed cell in $S$ precisely when this
  intersection includes exactly one half-space
  $M\in \beta\cup \iota(\beta)$ for each
  $\beta \in \Xi - (\Xi_P \cap \iota(\Xi_P))$.
\end{proof}

\subsection{Rank and height}

We define here the notion of rank and height of a general polyhedral
sets that will be used to define a normal series in the piecewise
isometry group $PI(\Delta)$ and allow for inductive arguments.

\begin{prop} \label{prop:reduction} Every nonempty polyhedral set $X$
  admits a finite decomposition into irreducibles.
\end{prop}

\begin{proof}
  It suffices to show every convex polyhedral set $P$ admits a finite
  decomposition into irreducibles. We make the following reduction:
  Fix hyperplanes $H_1,\dotsc, H_n\in \Omega$ representing each
  parallelism class of hyperplane. It suffices to decompose the
  essential intersection of $P$ with the closure of each connected
  component of $V - \bigcup_{i=1,\dotsc, n} H_i$. To that end, we will
  suppose that for each $\xi\in \Xi$, our set $P$ is $\xi$-bounded or
  $\iota(\xi)$-bounded.
  
  Let $\mathcal{Q}$ be the collection of nonempty pieces formed by
  intersecting $P$ with every thin $\xi$-layer, as defined
  following \cref{def:thin}, for which $P$ is both $\xi$-
  and $\iota(\xi)$-bounded. That is, we `chop up' $P$ as much as
  possible while slicing only in directions that intersect $P$ in
  finitely many hyperplanes. That is, $\mathcal{Q}$ is the collection
  of all minimal nonempty polyhedral sets of the form
  \[ P \cap
    \bigcap_{\xi\in \Xi_P \cap \iota(\Xi_P)} M_\xi \cap M_\xi' \text{
      where } M_\xi\in \xi \text{ and }M_\xi'\in \iota(\xi),
  \]
  where $\Xi_P$ is the set of $\xi\in \Xi$ such that $P$ is
  $\xi$-bounded. Clearly $\mathcal{Q}$ forms a finite decomposition of
  $P$. We claim that each $Q\in \mathcal{Q}$ is irreducible.

  To that end, consider any $Q\in \mathcal{Q}$. To see that $Q$ is
  irreducible, we check the two conditions of
  \cref{lem:irred-form}. The first condition of \cref{lem:irred-form}
  follows immediately from the fact that $Q$ is a subset of $P$, which
  satisfies the condition by the reduction at the start of the proof.
  To see that $Q$ is thin, suppose $Q$ is both $\xi$-bounded and
  $\iota(\xi)$-bounded for some $\xi\in \Xi$.  By
  \cref{lem:unboundeddirection} we know $L(Q) = L(P)$, so $Q$ and $P$
  shared the same boundedness properties, so
  $\xi \in \Xi_P\cap \iota(\Xi_P)$. Thinness of $Q$ now follows
  from minimality of its defining intersection.
\end{proof}

\begin{definition}
  The {\em rank} of a polyhedral set $X$, denoted $\rk(X)$, is the
  maximal rank among all convex polyhedral subsets $P\subset X$.
\end{definition}

\begin{prop} \label{prop:height-well-def} Suppose a polyhedral set $X\subset V$ has
  rank $k$. Then the number of irreducible convex polyhedral sets of
  rank $k$ in a finite decomposition of $X$ into irreducibles is
  independent of choice of finite decomposition into irreducibles.
\end{prop}
\begin{proof}
  Suppose $X$ has two finite decompositions into irreducible convex
  polyhedral sets $P_1,\dotsc, P_m$ and $Q_1,\dotsc, Q_n$, where the
  rank-$k$ irreducibles in the former are $P_1,\dotsc, P_d$ for
  $d\leq m$. For each $i=1,\dotsc, d$, we know that
  $P_i \cap Q_1, \dotsc, P_i \cap Q_n$ is a finite decomposition of
  $P_i$, so by irreducibility there exists a unique index
  $1\leq j_i \leq n$ such that $rk(P_i \cap Q_{j_i}) = k$. Since
  $\rk(P_i\cap Q_{j_i}) \leq \rk(Q_{j_i})\leq k$, we conclude that
  $\rk(Q_{j_i}) = k$. Therefore the assignment $i \mapsto j_i$ is a
  well-defined map from $\set{1,\dots, d}$ to the set of indices $j$
  for which $\rk(Q_j)=k$. An inverse function is determined by
  assigning to each such index $j$ the unique index $i_j$ such that
  $rk(Q_j \cap P_{i_j}) = k$, so the proof is complete.
\end{proof}

\begin{definition}
  The {\em height} of a polyhedral set $X\subset V$, denoted
  $\height(X)$, is the number of convex polyhedral sets of rank
  $\rk(X)$ in any finite decomposition of $X$ into irreducibles.
\end{definition}

\subsection{Alcoves}

Here we define irreducible convex polyhedral sets of
particularly nice form. Recall that a {\em tile} is a top-dimensional
cell of the tessellation $\Delta$. That is, a tile is a connected
component of $V - \bigcup_{H\in \Omega} H$.
\begin{definition}
  Given a tile $\tau \in \Delta$ and a spherical cell $\sigma \in
  L(\Delta)$, the {\em alcove} $A[\tau,\sigma]$ is the intersection of
  all half-spaces $M \in \mathcal{H}$ such that $\tau \subset M$ and
  $\sigma \subset L(M)$.
\end{definition}

Note that every tile $\tau$ is itself an alcove, $\tau = A[\tau,
\emptyset]$. 

\begin{lemma} \label{lem:alcove-basics} For any tile $\tau \in \Delta$
  and open cell $\sigma\in L(\Delta)$, the alcove $A[\tau, \sigma]$
  has limit set equal to the closure of $\sigma$ and is irreducible.
\end{lemma}
\begin{proof}
  Clearly $\sigma \subset L(A[\tau,\sigma])$, and certainly
  $L(A[\tau, \sigma])$ is closed, as it is the intersection of closed
  hemispheres. On the other hand, for any point $p\in \partial V$ not
  in the closure of $\sigma$, by definition of the cells of the
  spherical complex $L(\Delta)$ there is a half-space $M\in \mathcal{H}$
  such that $L(M)$ contains $p$ in its interior and is disjoint from
  $\sigma$. We may choose $M$ to be disjoint from $\tau$, so that the
  half-space $V-M$ contains $\tau$. Since $p$ is in the interior of
  $L(M)$, it does not belong to $L(V-M)$, and so $p\notin
  L(A[\tau,\sigma])$.

  To see that $A[\tau, \sigma]$ is irreducible, by
  \cref{lem:irred-form} it remains only to check that it is thin. This
  is clear; if $A[\tau, \sigma]$ is contained in an $\xi$-layer
  $M\cap M'$, where $M\in \xi$ and $M'\in \iota(\xi)$, then it is
  contained in the thin $\xi$-layer
  $b_\tau(\xi) \cap b_\tau(\iota(\xi))$.
\end{proof}

\begin{prop} \label{lpropem:alcove-decomp}
  Every polyhedral set $P$ admits a finite decomposition into alcoves.
\end{prop}
\begin{proof}
  Proceed by induction on the rank of $P$. Every polyhedral set of
  rank 0 is a finite union of tiles, which are themselves alcoves, so
  the base case is immediate.

  Now suppose $P$ has rank $k$. By \cref{prop:reduction}, it suffices
  to consider the case that $P$ is irreducible. Fix a tile
  $\tau \subset P$ and consider the alcove $A[\tau, L(P)]$, which is
  well-defined by \cref{lem:irred-form}. Cutting $P$ along the
  hyperplanes bounding $\tau$ produces a finite decomposition of $P$,
  which by \cref{prop:reduction} can be refined into a finite
  decomposition into irreducibles $Q_1, Q_2,\dotsc, Q_\ell$ where
  $Q_1 = A[\tau, L(P)]$. By \cref{prop:height-well-def}, each $Q_i$ has
  rank less than $k$ when $i>1$, so we are done by induction.
\end{proof}

\subsection{Canonical translations} \label{sec:endotrans}

Let $P\subset V$ be a convex polyhedral set. An {\em edge-ray} of $P$
is a ray in $V$ that is contained in a 1-dimensional face of $P$. If
$r$ is an edge-ray of $P$, then $L(r)$ is a vertex of the complex
$L(\Delta)$ that is an extremal point in the spherical convex set
$L(P)$. Conversely, if $v\in L(P)$ is an extremal point, then there is
an edge-ray $r$ in $P$ with $L(r) = v$.  It follows from the standard
characterization of a polyhedron as the sum of a polytope and its
recession cone \cite[1.12.ii]{Zieg} that each edge-ray $r$ of $P$ is
parallel to an extreme ray of its recession cone $\cone(P)$, and
conversely for each extreme ray of $\cone(P)$ there is a parallel
edge-ray of $P$.

For each edge-ray $r$ of $P$, by \cref{cor:line-trans} there is a
nontrivial translation in $T$ stabilizing the line $\cyc{r}$ spanned
by $r$. Let $t_r \in T$ be the unique element of $stab_T(\cyc{r})$ of
minimal translation length such that $r t_r\subset r$. (Recall that
$T$ acts on the right.) Since $t_r$ is translation by a vector $v_r$
in $\cone(P)$, we see $P t_r \subset P$. Now let $\Xi_r^+$ be the set
of all $\xi \in \Xi$ such that $r$ has nontrivial intersection with
every half-space $M\in \xi$. Equivalently, $\Xi_r^+$ is the set of
$\xi \in \Xi$ such that $L(r) \notin L(\iota(\xi))$. For each
$\xi \in \Xi_r^+$, there is a strict containment
$b_{P t_r}(\xi)\subsetneq b_P(\xi)$. On the other hand, for each
$\beta \in \Xi - (\Xi_r^+ \cup \iota(\Xi_r^+))$, we have
$b_{P t_r}(\beta) = b_{P}(\beta)$.

\begin{definition}
  For each irreducible convex polyhedral set $P$, the associated {\em
    canonical translation} $t_P\in T$ is the composition of all $t_r$
  as $r$ runs over representatives of each parallelism class of
  edge-rays of $P$.
\end{definition}
Since the limit set of an irreducible convex polyhedral set $P$ is a
cell $L(P)$ in the spherical complex $L(\Delta)$ by \cref{lem:irred-form},
its canonical translation $t_P$ is the composition of one translation
$t_r$ for each vertex of the cell $L(P)$. In particular, the canonical
translation $t_P$ depends only on the limit set $L(P)$. The canonical
translation sends $P$ `deeply into itself' in the following sense:

\begin{lemma} \label{lem:translation-inside} Suppose $P\subset V$ is
  an irreducible convex polyhedral set with canonical translation
  $t_P \in T$. For each convex polyhedral set $Q \subset P$ such that
  $L(P) = L(Q)$, there is a natural number $N$ such
  that $t_P^N(P) \subset Q$.
\end{lemma}
\begin{proof}
  Since $L(Q) = L(P)$, for each $\xi \in \Xi$ we know $P$ is
  $\iota(\xi)$-unbounded if and only if $Q$ is
  $\iota(\xi)$-unbounded. For each $\xi \in \Xi$ such that $Q$ is
  $\iota(\xi)$-unbounded, it follows from \cref{lem:unboundeddirection}
  that there is an edge-ray $r$ of $P$ such that $\xi \in \Xi_r^+$. We
  therefore have a descending sequence
  $b_P(\xi) \supsetneq b_{P t_r}(\xi) \supsetneq b_{P t^2_r}(\xi)
  \supsetneq \dotsb$, from which we conclude there is a descending
  sequence
  \[
    b_P(\xi) \supsetneq b_{P t_P}(\xi) \supsetneq
    b_{P t^2_P}(\xi) \supsetneq \dotsb
  \]
  Since there are finitely many half-spaces between $b_P(\xi)$ and
  $b_Q(\xi)$, there is a natural number $n_\xi$ such that
  $b_{P t_r^{n_\xi}}(\xi) \subset b_Q(\xi)$. Let $N$ be the
  maximum value of $n_\xi$ where $Q$ is
  $\iota(\xi)$-unbounded. Then by construction we have
  $P t_P^N \subset Q$, which completes the proof.
\end{proof}

\subsection{Germs}

In this section we will define the equivalence classes of irreducible
convex polyhedral sets, called germs, and prove that the piecewise
isometry group acts on the set of germs.

\begin{definition}
  Suppose $P$ and $Q$ are nonempty irreducible convex polyhedral
  sets. Say that $P$ and $Q$ are {\em commensurable}, written
  $P\sim Q$, if $P\doublecap Q \neq \emptyset$ and $L(P) = L(Q)$.
\end{definition}

The following alternative characterizations of commensurability will
prove useful. Recall that the rank of the empty polyhedral set is $-1$.

\begin{lemma} \label{lem:commens-characterizations}
  Given nonempty irreducible convex polyhedral sets $P$ and $Q$ , the
  following are equivalent:
  \begin{enumerate}
  \item $P$ and $Q$ are commensurable,
  \item $L(P) = L(P\doublecap Q) = L(Q)$,
  \item $rk(P) = rk(P\doublecap Q) = rk(Q)$.
\end{enumerate}
\end{lemma}
\begin{proof}
  We prove the first two conditions imply each other, and the latter
  two conditions imply each other.

  Clearly the second condition implies the first. On the other hand,
  suppose $P$ and $Q$ are commensurable. We clearly have
  $L(P\doublecap Q) \subset L(P)$. Conversely, any direction in $L(P)$
  is represented by a ray $r$ contained in $P$ originating at a point
  in $P\doublecap Q$. Since $L(P) = L(Q)$ this ray is also contained
  in $Q$, hence in $P\doublecap Q$. This proves
  $L(P) = L(P\doublecap Q)$, and the equality
  $L(P\doublecap Q) = L(Q)$ follows analogously.

  Now for the latter two conditions. Since $rk(P) = dim(L(P))+1$,
  clearly the equalities $L(P) = L(P\doublecap Q) = L(Q)$ imply
  $rk(P) = rk(P\doublecap Q) = rk(Q)$. On the other hand, suppose
  $rk(P) = rk(P\doublecap Q) = rk(Q)$. Since $L(P\doublecap Q)$ is a
  subcomplex of the single cell $L(P) \subset L(\Delta)$ whose
  codimension is equal to the corank of $P\doublecap Q \subset P$, the
  condition that $rk(P) = rk(P\doublecap Q)$ directly implies
  $L(P) = L(P\doublecap Q)$. We analogously see that
  $L(P\doublecap Q) = L(Q)$, completing the proof.
\end{proof}

\begin{lemma} \label{lem:comm-equiv-rel}
  Commensurability is an equivalence relation on the set of
  irreducible convex polyhedral sets.
\end{lemma}
\begin{proof}
  Symmetry and reflexivity are clear. To see that transitivity holds,
  we use the second characterization of commensurability in
  \cref{lem:commens-characterizations}.  Suppose that irreducible
  convex polyhedral sets $P,Q,R$ satisfy
  $L(P) = L(P\doublecap Q) =L(Q)= L(Q\doublecap R)= L(R)$. Since
  $L(P)=L(Q)=L(R)$, the canonical translations $t_P$, $t_Q$, and $t_R$
  as defined in \cref{sec:endotrans} are all equal to the same
  translation $t\in T$. By \cref{lem:translation-inside}, there are
  natural numbers $m$ and $n$ such that $t^m(P) \subset Q$ and
  $t^n(Q) \subset R$. It follows that $t^{m+n}(P) \subset P\cap R$, so
  $L(P) = L(P\doublecap R)$. Similarly noting
  $L(P\doublecap R) = L(R)$, the proof is complete.
\end{proof}

\begin{definition}
  A {\em germ} is a commensurability class of irreducible convex
  polyhedral sets. The commensurability class of an irreducible convex
  polyhedral set $C$ is written $[C]$.
  \begin{itemize}
    \item The set of all germs is $\Gamma$.
    \item For each $k\geq 0$, the set of all equivalence classes of
      irreducibles of rank $k$ is $\Gamma^k$.
    \item For any polyhedral set $P$,
      let $\Gamma^k(P)$ be the set of germs $[C]$ where $C$ is an
      irreducible convex polyhedral subset of $P$.
      \end{itemize}
\end{definition}

\begin{lemma} \label{lem:pi-alcove-characterization} Let $P$ be a
  polyhedral set and $f:P\to V$ an essentially injective map. Then $f$
  is a piecewise isometric embedding if and only if for each
  irreducible convex polyhedral subset $C\subset P$, there is some
  natural number $m$ such that the restriction of $f$ to $C t_C^m$ is
  isometric.
\end{lemma}
\begin{proof}
  Suppose $f$ is a piecewise isometric embedding, and let $C\subset P$
  be irreducible. Choose a finite decomposition of $P$ into
  irreducibles $Q_1,\dotsc, Q_k$ such that the restriction of $f$ to
  each $Q_i$ is isometric. There is a unique index $j$ such that
  $C\cap Q_j$ has rank equal to that of $C$. By
  \cref{lem:translation-inside} there is some $m$ such that $C t_C^m$
  is a subset of $Q_j$, so $f$ is isometric on $C t_C^m$.

  Conversely, suppose for each irreducible $C\subset P$ there is some
  $m$ such that $f$ is isometric on $C t_C^m$. We prove that $f$ is
  piecewise isometric by induction on the rank of $P$. When $P$ has
  rank $0$, each irreducible $C\subset P$ is a single tile and so
  $t_C$ is trivial. Therefore $f$ is piecewise isometric. For the
  inductive step, let $Q_1,\dotsc, Q_k$ be any finite decomposition of
  $P$ into irreducibles. For each irreducible $Q_j$ of top rank, find
  $m_j$ such that $f$ is isometric on $Q_j t_{Q_j}^{m_j}$. The
  complement of the union of these translations in $P$ is a
  sub-polyhedral set of lower rank by \cref{prop:height-well-def}, so
  we are done by induction.
\end{proof}

There is a natural action of $PI(\Delta)$ on $\Gamma$ preserving rank:
Given a germ $\gamma\in \Gamma$ and $f\in PI(\Delta)$, choose a
representative irreducible $P$ such that $\gamma = [P]$ and by
\cref{lem:pi-alcove-characterization} a natural number $m$ such that
$f$ is isometric on $P t_P^m$. Define the right action
\[
  \gamma f_\ast = [ P t_P^m f ].
\]

\begin{lemma}
  The action of $PI(\Delta)$ on $\Gamma$ is well-defined.
\end{lemma}
\begin{proof}
  Fix any $f\in PI(\Delta)$ and $\gamma\in \Gamma$. First note that,
  for a given representative $\gamma = [P]$, the germ
  $[ P t_P^m f ]$ is independent of choice of $m$, since if $m<n$
  then $P  t_P^n f$ is a subset of $P t_P^m f$ with the same
  limit set. Now suppose irreducibles $P$ and $Q$ both represent
  $\gamma$. Then $L(P) = L(Q)$ by definition, so $P$ and $Q$ have the
  same canonical translation $t\in T$. Find $m,n$ such that $f$ is
  isometric on both $P t^m$ and $Q t^n$. By
  \cref{lem:translation-inside} there are $m', n'$ such that
  $P t^{m+m'} \subset Q t^n$ and $Q t^{n+n'} \subset P t^m$. We
  then have
  \[ \gamma f_\ast = [P t^m f ] = [P t^{m+m'} f ] = [ Q t^n f ].
  \]
\end{proof}

\subsection{A normal series in the piecewise isometry group}

The action of $PI(\Delta)$ on the set of germs of irreducible pieces is
used to define a normal series in $PI(\Delta)$, in \cref{prop:filtration} below.

\begin{definition} \label{def:ranksuppheight}
  For a piecewise isometry $f\in PI(\Delta)$,
  \begin{enumerate}
  \item The {\em support} of $f$ is the union of all convex polyhedral
    sets $P$ such that $f$ restricts to a nontrivial isometry on the
    interior of $P$, denoted $\supp(f)$;
  \item the {\em rank} of $f$ is the greatest rank of a
    convex polyhedral set $P\subset \supp(f)$, denoted $\rk(f)$;
  \item the {\em height} of $f$ is the number of irreducible pieces of
    rank $\rk(f)$ in any finite decomposition of $\supp(f)$ into
    irreducibles, denoted $\height(f)$.
  \end{enumerate}
\end{definition}

Given any finite decomposition $P_1,\dotsc, P_n$ of $V$ into subsets
on which $f$ is isometric, each $P_i$ is either contained in
$\supp(f)$ or contained in the closure of $V - \supp(f)$, so the
support of $f$ is a polyhedral set. The rank and height of $f$ are by
definition the rank and height of its support.

For each $r=0,1,\dotsc, \dim(V)$ define
\[
  G_r = \set{ f\in PI(\Delta) \suchthat \rk(f) \leq r}
\]
and
\[
  C_r = \set{ f\in PI(\Delta) \suchthat \gamma f_\ast = \gamma \text{ for all } \gamma\in
    \Gamma^r}.
\]
It is straightforward to see that each $G_r$ and each $C_r$ is a normal subgroup of
$PI(\Delta)$.

\begin{lemma} \label{lem:Grform} $G_r$ is the group of all
  $f\in PI(\Delta)$ such that each irreducible convex polyhedral set
  $P$ of rank strictly greater than $r$ is commensurable with an
  irreducible convex polyhedral set $Q$ on which $f$ is trivial.
\end{lemma}
\begin{proof}
  This follows from \cref{lem:translation-inside} and
 \cref{lem:pi-alcove-characterization}. 
\end{proof}

Clearly $G_{r-1}\leq C_r$: if $f\in G_{r-1}$ then, for any
$\gamma \in \Gamma^r$, by \cref{lem:Grform} there is an irreducible $Q$
representing $\gamma$ such that $f$ is trivial on $Q$, so
$\gamma f_\ast = \gamma$.
  
\begin{lemma} \label{lem:firstfiltration}
  For each $0 < r \leq \dim(V)$ we have $C_r \leq G_r$.
\end{lemma}
\begin{proof}
  Suppose $f\in PI(\Delta) - G_r$. This means that there is an
  irreducible piece $P$ of rank $r+1$ such that $f$ is a nontrivial
  isometry on $P$. We may take $P$ to be an alcove
  $P = A[\tau, \sigma]$. Consider the image of the base tile
  $\tau$. There are two cases. First, if $\tau f \neq \tau$, then
  there is a hyperplane parallel to a face of $P$ separating the two
  tiles. Cutting along a hyperplane in this family produces a corank-1
  sub-convex set of $P$ which is moved by $f$, so $f\notin
  C_r$. Second, if $\tau f = \tau$ then $f$ has nontrivial
  differential $df$, which we consider as a linear map $V\to V$. Use
  this to find a corank-1 sub-convex set $Q\subset P$ whose recession
  cone contains a nonzero vector $v$ and does not contain $\pm
  df(v)$. Then $[Q]f_\ast \neq [Q]$, so $f\notin C_r$.
\end{proof}

We now further refine the normal series
$1\leq C_0 \leq G_0 \leq C_1\leq G_1\leq  \dotsb \leq PI(\Delta)$. For each germ
$\gamma \in \Gamma$ define the subgroup
\[C(\gamma) = \set{ f\in PI(\Delta) \suchthat \gamma f_\ast = \gamma}
  \leq PI(\Delta),
\]
so that $C_r$ is the intersection $C(\gamma)$ over all
$\gamma \in \Gamma^r$.

For each $\xi \in \Xi$, let $v_\xi$ be the positively-oriented unit
vector normal to the boundary of $M$. Let $\matrixO(V,\Xi)$ be the
finite subgroup of $\matrixO(V)$ preserving the set of vectors
$\{v_\xi\}_{\xi\in \Xi}$. If $f$ is an isometry of $V$ preserving
$\Delta$, then its differential $d_xf$ lies in $\matrixO(V,\Xi)$ for
any point $x\in V$, and is independent of $x$. It follows that for each
$\gamma\in \Gamma$, there is a well-defined homomorphism
\[
  \bar\varphi_\gamma : C(\gamma) \to \matrixO(V,\Xi)
\]
determined by choosing any irreducible convex polyhedral set $P$
representing $\gamma$ on which $f$ is isometric, and setting
$\bar\varphi_\gamma(f) = d \restr{f}{P}$.

For each $\gamma \in \Gamma$ define $C^+(\gamma) =
\ker(\bar\varphi_\gamma)$, so that $C^+(\gamma) \leq C(\gamma)$ is a
subgroup of finite index. Further define $C^+_r \leq C_r$ to be the
intersection of $C^+(\gamma)$ over all $\gamma \in \Gamma^r$. 
If $f\in G_{r-1}$ and $\gamma \in \Gamma^r$ then by
\cref{lem:Grform} there is an irreducible $Q$ representing
$\gamma$ such that $f$ is trivial on $Q$, so $G_{r-1} \leq C^+_r$. It
is straightforward to see that each $C^+_r$ is normal in
$PI(\Xi)$. We summarize this discussion:

\begin{prop} \label{prop:filtration} For each $0\leq r \leq \dim(V)$,
  the subgroups $G_r, C_r, C^+_r$ defined above are normal in
  $PI(\Delta)$, and for each $r>0$ there are containments
  \[
    G_{r-1} \leq C^+_r \leq C_r \leq G_r.
  \]
  \qed
\end{prop}

\subsection{Stabilizer quotients} \label{sec:sequences}

We now study the successive quotients of the normal series given in
\cref{prop:filtration}.

\begin{lemma} \label{lem:gr-mod-cr}
  There is an exact sequence
  \[
    1 \to C_r \to G_r \to \Perm_f(\Gamma^r).
  \]
\end{lemma}
\begin{proof}
For each $r$, the action of piecewise isometries on germs determines a
homomorphism $G_r \to \Perm_f(\Gamma^r)$ whose kernel is $C_r\cap
G_r$, which is equal to $C_r$ by \cref{lem:firstfiltration}. 
\end{proof}

\begin{lemma} \label{lem:cr-mod-crplus}
  There is an exact sequence
  \[
    1 \to C^+_r \to C_r \to \bigoplus_{\gamma \in \Gamma^r}
    \matrixO(V,\Xi).
  \]
\end{lemma}
\begin{proof}
  The product of the homomorphisms $\bar \varphi_\gamma$ is the
  rightmost homomorphism. The codomain is a direct sum since any $f\in
  C_r$ is an element of $G_r$ by \cref{lem:firstfiltration}, and
  its support contains finitely many germs of rank $r$. The kernel is
  $C^+_r$ by definition.
\end{proof}

\begin{lemma} \label{lem:crplus-mod-gr}
  There is an exact sequence
  \[
    1 \to G_{r-1} \to C_r^+ \to \bigoplus_{\gamma \in \Gamma^r} T
  \]
\end{lemma}
\begin{proof}
  Given $f\in C^+_r$ and $\gamma\in \Gamma^r$, choose an irreducible
  convex polyhedral set $C$ representing $\gamma$ on which $f$ is
  isometric by \cref{lem:pi-alcove-characterization}. The restriction
  of $f$ to $C$ is a translation in $T \leq \Isom(\Delta)$ since it
  extends to an isometry of $V$ with trivial differential. The element
  of $T$ is independent of choice of representative $C$ since any two
  choices have intersection with nonempty interior by definition, and
  it is similarly easy to see that this defines a homomorphism. The
  kernel is $G_{r-1}$ by \cref{lem:Grform}.
\end{proof}

\begin{theorem}
  The group $PI(\Delta)$ is elementary amenable.
\end{theorem}
\begin{proof}
  The quotients $G_r/ C_r$ and $C_r / C_r^+$ are locally finite by
  \cref{lem:gr-mod-cr} and \cref{lem:cr-mod-crplus}, therefore
  elementary amenable. The quotients $C_r^+ / G_{r-1}$ are abelian by
  \cref{lem:crplus-mod-gr}, and therefore elementary amenable. It
  follows then from \cref{prop:filtration} that $PI(\Delta)$, which is
  equal to $G_{\dim(V)}$, is an iterated extension of elementary
  amenable groups by elementary amenable groups, and is therefore
  elementary amenable.
\end{proof}

Future work will identify the images of the rightmost maps in the
above exact sequences, and use this to prove finite generation and
other finiteness properties of $PI(\Delta)$.

\section{Examples} \label{sec:examples}

\subsection{Tessellations associated to affine Weyl groups}

Suppose $\Phi$ is a crystallographic root system in $\R^n$ with inner product $(u,v) = u\cdot v$. 
For our notation and conventions we follow Humphreys \cite{Hum}.  For each
root $\alpha\in \Phi$ and integer $k\in \Z$, consider the hyperplane
$H_{\alpha, k} = \set{v\in \R^n \suchthat (v,\alpha) = k}$. The set
$\Omega$ of all such hyperplanes determines a cocompact tessellation
$\Delta$ of $V$ satisfying our standing assumption.

Suppose the simple roots are $\alpha_1,\dotsc, \alpha_n \in
\Phi$. This determines a set of positive roots $\Phi^+ \subset
\Phi$. The affine Weyl group $\Waff[]$, which is generated by reflections
over the hyperplanes in $\Omega$, acts simply transitively on tiles of
$\Delta$, so that every convex polyhedral set contains a tile in the
orbit of the fundamental tile
\[
  \tau_0 := \set{v\in \R^n \suchthat 0< v\cdot \alpha < 1\text{ for
      all } \alpha\in \Phi^+}.
\]

The set of germs $\Gamma$ admits a nice description in this case,
which we give in \cref{cor:germclassification} and
\cref{cor:germclassification2}. For each subset
$I\subset \set{\alpha_1,\dotsc, \alpha_n}$, let $V_I \subset \R^n$ be
the linear span of $I$ and set $\Phi_I = \Phi \cap V_I$, so that
$\Phi_I$ is a crystallographic root system in $V_I$. For each such
$I$, define the {\em standard convex polyhedral set}
\[
  P_I = \set{ v\in \R^n \suchthat v\cdot \alpha_i \geq 0 \text{ for }
    i=1,\dotsc, n \text{ and } v\cdot \alpha \leq 1 \text{ for }\alpha\in
    \Phi_I^+ }.
\]
Then each set $P_I$ is irreducible of rank $n-\abs{I}$, and the limit
set $L(P_I)$ is a closed simplex in the geometric realization of the
Coxeter complex associated to $\Phi$. For example,
$P_{\set{\alpha_1,\dotsc, \alpha_n}}$ is equal to the fundamental tile
$\tau_0$, and $P_\emptyset$ is equal to the fundamental alcove for the
action of the finite Weyl group $W$. The sets $P_I$ are precisely the
alcoves $A[\tau_0, \sigma]$ where $\sigma$ is a cell of the Coxeter
complex which is a face of the simplex corresponding to the
fundamental alcove for the finite Weyl group.

Let $W$ be the finite Weyl group associated to $\Phi$, and $W_I$ and
$\Waff[I]$ the finite and affine Weyl groups of the root system
$\Phi_I$ in $V_i$. Each group acts on the set of irreducible convex
polyhedral sets and induces an action on the spherical complex
$L(\Delta)$. Each closed cell of the Coxeter complex $L(\Delta)$ is of
the form $L(P_I) w$ for some $w\in W$ and
$I\subset \set{\alpha_1,\dotsc, \alpha_n}$.

\begin{prop} \label{prop:irreds-with-same-limit}
  Suppose irreducible convex polyhedral sets $C$ and $D$ satisfy $L(C)
  = L(D) = L(P_I) w$. Then there is a unique element $\rho \in
  \Waff[I]$ such that $Cw^{-1}\rho w$ is commensurable with $D$.
\end{prop}
\begin{proof}
  We prove this in the case that $w$ is trivial; the general case
  follows. Let $T_C$ and $T_D$ be the orthogonal projections of $C$
  and $D$ to $V_I$. From the assumption that $L(C) = L(D) = L(P_I)$ we
  know that $T_C$ and $T_D$ are bounded subsets of $V_I$, and by
  irreducibility of $C$ and $D$ it follows that $T_C$ and $T_D$ are
  each the closure of a tile of the tessellation of $V_I$. Then there
  is a unique element $\rho \in \Waff[I]$ such that $T_C \rho =
  T_D$. It follows that $C \rho \cap D$ has nonempty interior. Since
  $L(C \rho) = L(C) = L(D)$, we conclude that $C \rho$ is
  commensurable with $D$.
\end{proof}
\begin{corollary} \label{cor:germclassification}
  For any $r\geq 0$ we have
  \[
    \Gamma^r = \bigcup_{\substack{I\subset \set{\alpha_1,\dotsc,
          \alpha_n}\\ \dim(V_I) = n-r}} \left(
      \bigcup_{[w]\in W/W_I} \left( \bigcup_{\rho \in \Waff[I]}
        [P_I \rho w] \right) \right)
  \]
  Moreover, each of these unions is disjoint.
\end{corollary}
\begin{proof}
  Suppose $P$ is an irreducible convex polyhedral set. Since the limit
  set of an irreducible is a simplex in the Coxeter complex, we know
  $L(P) = L(P_I) w$ for some $I\subset S$ and $w\in W$. By
  \cref{prop:irreds-with-same-limit} we conclude there is some
  $\rho \in \Waff[I]$ such that $P \sim P_I \rho w$.  To see that the
  unions are disjoint, suppose $Q_1\sim Q_2$ where
  $Q_1 = P_I \rho_1 w_1$ and $Q_2 = P_J \rho_2 w_2$ for
  $w_1, w_2 \in W$, $I,J\subset \set{\alpha_1,\dotsc, \alpha_n}$,
  $\rho_1 \in \Waff[I]$, and $\rho_2 \in \Waff[J]$. Then
  $L(P_I) w_1 = L(Q_1) = L(Q_2) = L(P_J) w_2$, and so it follows from
  the structure of the Coxeter complex that $I = J$ and
  $w_1 w_2^{-1} \in W_I$. Disjointness now follows from the uniqueness
  statement in \cref{prop:irreds-with-same-limit}.
\end{proof}

Since each simplex of the Coxeter complex is identified with
$L(P_I) w$ for some $w\in W$ and
$I\subset \set{\alpha_1,\dotsc,\alpha_n}$, and the stabilizer of
$L(P_I) w$ under the action of $\Waff[]$ is equal to
$w^{-1} \Waff[I] w$, we have the following classification of germs in
terms of the action of the affine Weyl group on the Coxeter complex
$L(\Delta)$:

\begin{corollary} \label{cor:germclassification2}
  The set of germs can be written as a disjoint union
  \[
    \Gamma = \bigcup_{\sigma:= wL(P_I)\subset L(\Delta)} \left(
      \bigcup_{\rho \in stab_{\Waff[]}(\sigma)}  [P_I \rho w] \right)
  \]
\end{corollary}

For example, each top-dimensional simplex of $L(\Delta)$ has trivial
stabilizer under the action of $\Waff[]$, so the top-rank germs
$\Gamma^n$ are in bijection with top-dimensional simplices of
$L(\Delta)$, which are in bijection with the finite Coxeter group
$W$. At the other extreme, the empty simplex is trivially stabilized by
all of $\Waff[]$, and the fact that the germs of rank $0$ are in
bijection with $\Waff[]$ is simply a reflection of the fact that the
affine Weyl group acts simply transitively on tiles.

Germs of rank $r$ are organized into parallel families, one family for
each cell $L(P_I) w$ in the Coxeter complex, where
$I\subset \set{\alpha_1,\dotsc,\alpha_n}$ is a subset of size
$n-r$. The number of such families associated to $I$ is
$\abs{W} / \abs{W_I}$ by the orbit-stabilizer theorem. The argument in
the proof of \cref{prop:irreds-with-same-limit} shows that
an irreducible polyhedral convex set with limit set $L(P_I)w$ has
cross-sections isometric to a tile in the induced tessellation of
$V_I$. This shows that o in general that there are irreducible convex polyhedral sets
that are not commensurable to isometric irreducibles.

\bold{Canonical translations.} Consider a given edge-ray $r$ of the
cone $P_\emptyset$. Then $r$ lies transverse to a unique facet of
$P_\emptyset$ lying in a hyperplane $H_{\alpha_i,0}$ for some
$\alpha_i$. The translation $t_r$ along $r$ is given by the
fundamental coweight $\omega^\vee_i$, determined by
$\omega^\vee_i \cdot \alpha_j = \delta_{ij}$. Since every vertex in
$L(\Delta)$ is obtained from a vertex of the base simplex by applying
an element of the finite Weyl group, every edge-ray translation
translates by a coweight $\omega^\vee_i w$ for some $w\in W$ and
fundamental coweight $\omega^\vee_i$.

\bold{Irreducible root systems.} Every crystallographic root system is
a product of irreducible root systems, which are classified by type
$X_n$ belong to one of the families $A_n$, $B_n$, $C_n$ or $D_n$, or
one of five exceptional cases. Suppose $X_n$ is irreducible with
simple roots $\a_1, \cdots, \a_n$ and $n\times n$ Cartan matrix
$$\cC = \cC(X_n) = \left[ \frac{2\a_i\cdot\a_j}{\a_j\cdot\a_j} \right]
= [c_{ij}]\quad\hbox{with inverse}\quad \cC^{-1} = D = [d_{ij}].$$
The finite Weyl group is $W(X_n)$, and the associated affine Weyl
group is $W(X_n^+)$, where $X_n^+$ is the affine root system obtained
from the associated affine Kac-Moody algebra.

There is a unique highest root $\theta \in X_n$, and the fundamental
tile is given by
\[
  \tau_0 = \set{v\in \R^n \suchthat 0< v\cdot \alpha_i \text{ for }
    i=1,\dotsc,n \text{ and } v\cdot \theta < 1}.
\]
In this case, the closure of $\tau_0$ is a simplex whose vertices are
$0$ and multiples of $\lambda_1,\dotsc, \lambda_n$, the fundamental
weights of $X_n$ given by
$$\lam_j = \sum_{i=1}^n d_{ij} \a_i, \quad\hbox{for  } 1\leq j\leq
n.$$

\subsection{Equilateral triangle tessellation of $\R^2$ from root system of type $A_2$}

The irreducible root system $A_2$ has two simple roots $\alpha_1$ and
$\alpha_2$ of length $\sqrt{2}$ at angle $2\pi/3$.  Associated to
$A_2$ is the tessellation $\Delta$ of $\R^2$ by equilateral triangles.
The associated affine Weyl group $W(A_2^+)$ is the triangle group,
$T(3,3,3)$, generated by the reflections in the three sides satisfying
$|r_1 r_2| = 3$, $|r_1 r_3| = 3$, $|r_2 r_3| = 3 $ so each pair
generates a permutation group, $S_3$.

The spherical complex $L(\Delta)$ on the circle at infinity consists of six
vertices, one antipodal pair for each of the three families of
parallel lines, and six $1$-cells. Those $1$-cells are the limit sets
of the six alcoves for the action of the finite Weyl group $S_3$, each
of which is an irreducible convex polyhedral set of rank $2$ in
$\Delta$. By \cref{cor:germclassification2}, each vertex in $L(\Delta)$
corresponds to a $D_\infty$-family of rank $1$ germs, each of which
can be represented by a beam in the direction of a ray, bounded by a
pair of adjacent parallel lines.

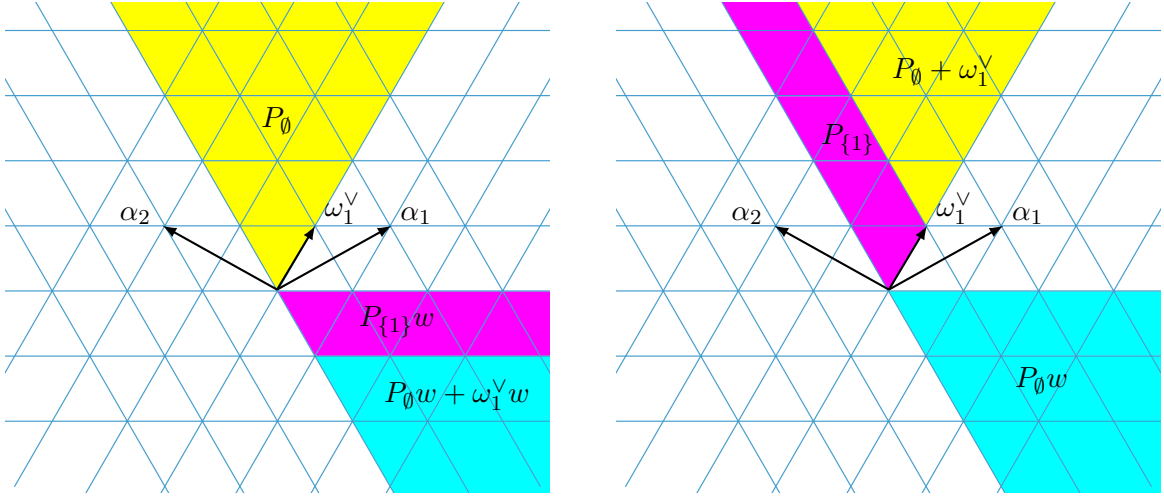
\begin{figure}[htb] 
  \centering
  \begin{subfigure}[b]{0.45\textwidth}
      \begin{tikzpicture}
        \node[anchor=south west, inner sep=0] (image) at (0,0)
        {\includegraphics[width=\textwidth]{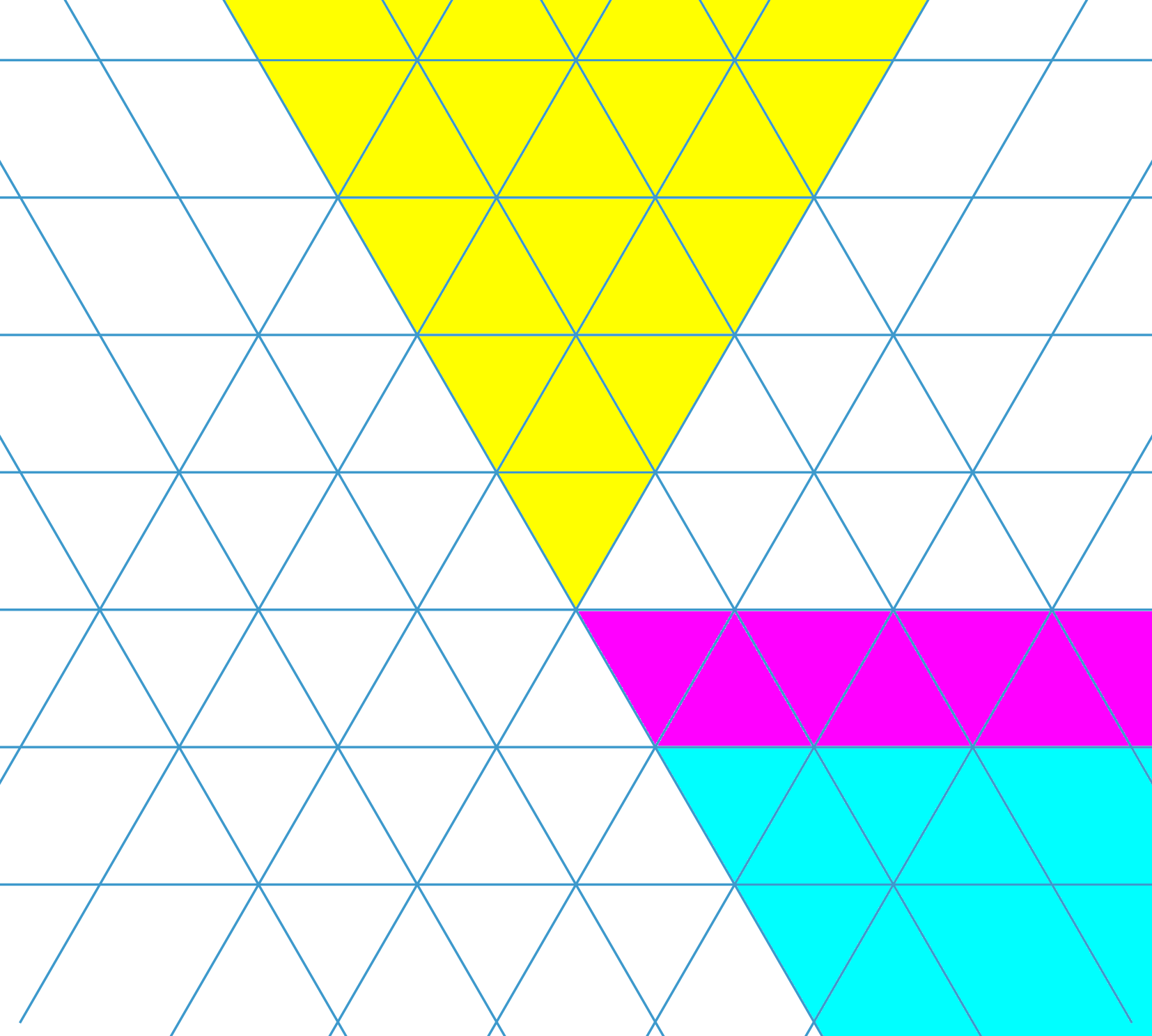}};
        
        \begin{scope}[x={(image.south east)},y={(image.north west)}]

          \node[color=black] at (0.5, 0.76) {$P_{\emptyset}$};
          \node[color=black] at (0.72, 0.35) {$P_{\{1\}}w$};
          \node[color=black] at (0.83, 0.2)
          {$P_{\emptyset}w+\omega_1^\vee w$};
          
          \node[color=black] at (0.755, 0.565) {$\alpha_1$};
          \node[color=black] at (0.24, 0.565) {$\alpha_2$};
          \node[color=black] at (0.62, 0.58) {$\omega^\vee_1$};
          \draw[-latex, thick] (0.5, 0.415) -- (0.71, 0.545); 
          \draw[-latex, thick] (0.5, 0.415) -- (0.29, 0.545); 
          \draw[-latex, thick,black] (0.5, 0.415) -- (0.57, 0.545); 

        \end{scope}
      \end{tikzpicture}
    \end{subfigure}
    \qquad
    \begin{subfigure}[b]{0.45\textwidth}
      \begin{tikzpicture}
        \node[anchor=south west, inner sep=0] (image) at (0,0)
        {\includegraphics[width=\textwidth]{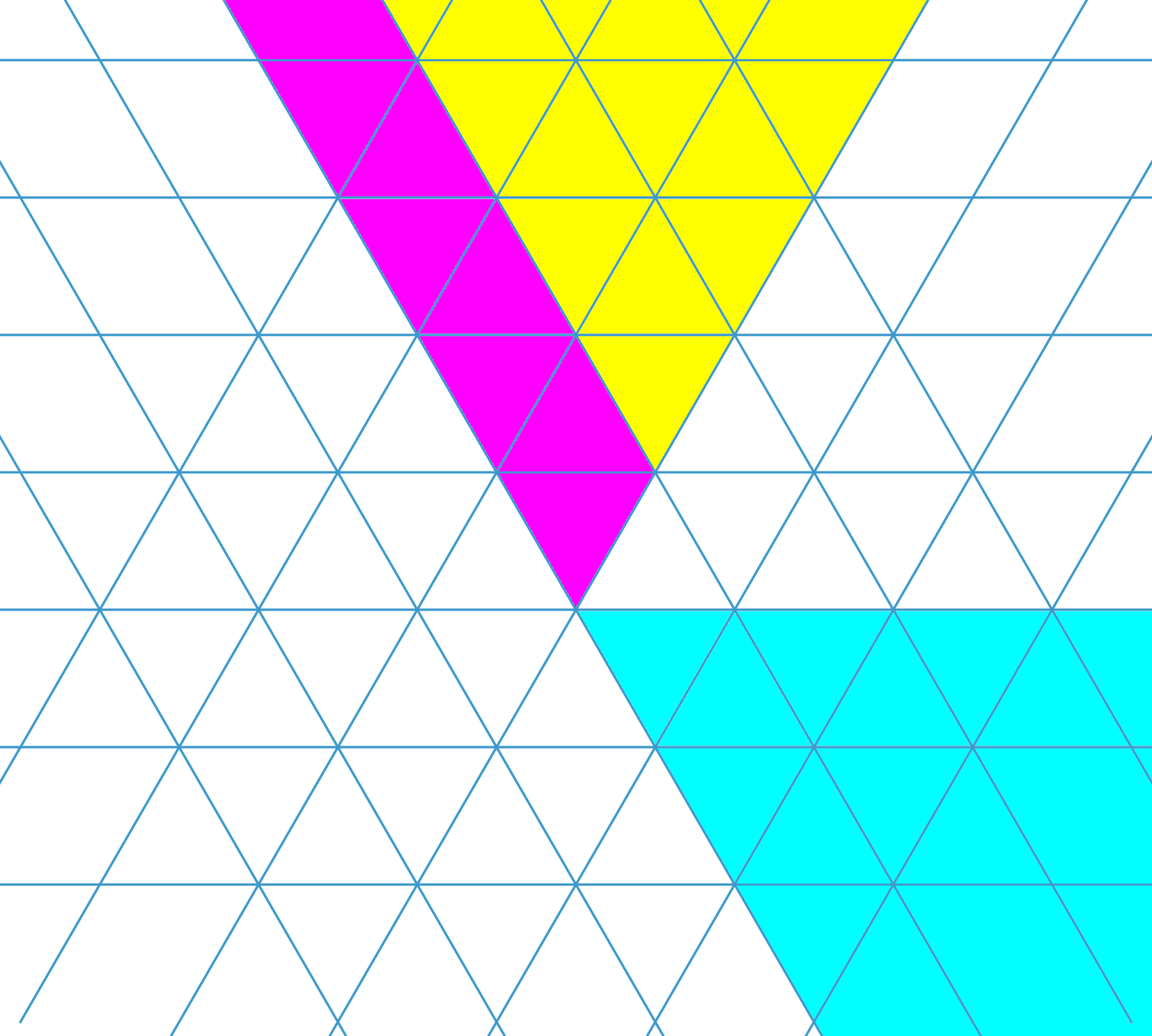}};
        
        \begin{scope}[x={(image.south east)},y={(image.north west)}]
          
          \node[color=black] at (0.6, 0.86) {$P_{\emptyset} + \omega_1^\vee$};
          \node[color=black] at (0.425, 0.72) {$P_{\{1\}}$};
          \node[color=black] at (0.78, 0.23) {$P_{\emptyset}w$};

          \node[color=black] at (0.755, 0.565) {$\alpha_1$};
          \node[color=black] at (0.24, 0.565) {$\alpha_2$};
          \node[color=black] at (0.62, 0.58) {$\omega^\vee_1$};
          \draw[-latex, thick] (0.5, 0.415) -- (0.71, 0.545); 
          \draw[-latex, thick] (0.5, 0.415) -- (0.29, 0.545); 
          \draw[-latex, thick,black] (0.5, 0.415) -- (0.57, 0.545); 
        \end{scope}
      \end{tikzpicture}
      \end{subfigure}
 \caption{\label{fig:PIA2} The domain and codomain finite
  decompositions of the plane for the piecewise isometry $f_3\in C_2^+
  - G_1$ when $w\in W(A_2)$ is rotation by $2\pi/3$}
\end{figure}

We use this example to illustrate examples of piecewise isometries
exhibiting a variety of placements in the normal series described in
\cref{sec:sequences}. 
\begin{itemize}
  \item An example of an element $f_1 \in G_2 - C_2$ is a piecewise isometry which
isometrically exchanges two alcoves of the finite Weyl group,
$P_\emptyset$ and $P_\emptyset w$, and is the identity outside of
$P_\emptyset \cup P_\emptyset w$.
\item An example of an element 
$f_2\in C_2 - C_2^+$ is the piecewise isometry which isometrically sends
$P_\emptyset$ to itself by an internal reflection, and is the identity
outside of $P_\emptyset$.
\item An example of an element $f_3 \in C_2^+ - G_1$, shown in 
  \cref{fig:PIA2}, is the piecewise isometry which, for a fixed
  nontrivial element $w\in W$,
\begin{enumerate}
\item translates $P_\emptyset$ into itself by the fundamental coweight
  $\omega^\vee_1$,
\item translates $P_\emptyset w -  P_{\{1\}}w$ isometrically onto
  $P_\emptyset w$ by $- \omega^\vee_1 w$,
\item isometrically maps $P_{\{1\}} w$ to $P_{\{1\}}$ via
  $w^{-1}$, and
\item is the identity outside of
  $P_\emptyset \cup P_{\emptyset} w$.
\end{enumerate}
\end{itemize}

\subsection{Tetrahedron tessellation of $\R^3$ from the affine Weyl group of type $A_3$}

Next we examine the tessellation of $\R^3$ by a tetrahedron which is
the fundamental domain for the affine Weyl group, $W(A_3^+)$.  In this
example, the fundamental tile $\tau_0$ is a tetrahedron with one pair
of opposite edges of length $1$ and the other four edges of length
$\sqrt{3}/2$. There are $24$ copies of $\tau_0$ meeting at each vertex
of the tessellation, corresponding to the finite subgroup, $S_4$, the
Weyl group $W(A_3)$. The union of the $24$ copies meeting at the
origin forms a rhombic dodecahedron, whose boundary is a polygonal
complex isomorphic to the Coxeter complex $L(\Delta)$. See 
\cref{fig:RDandA3Sphere}.

\begin{figure}[htb]
  \centering
  \includegraphics[scale=0.7]{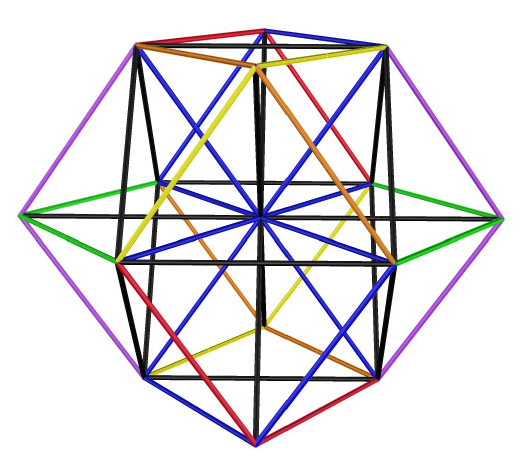}
  \qquad
  \includegraphics[scale=0.6]{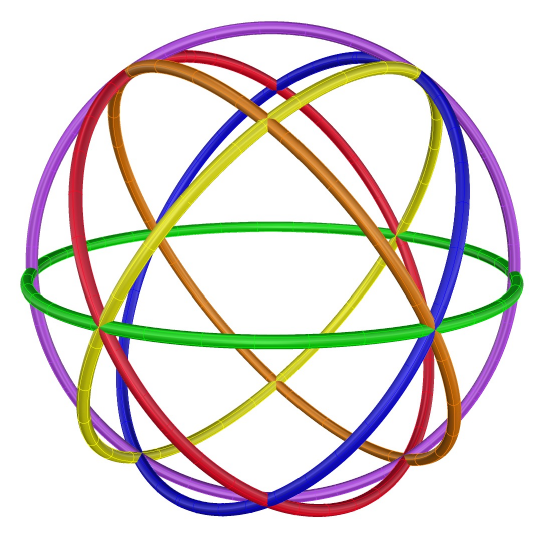}
\caption{\label{fig:RDandA3Sphere} Affine $A_3$ rhombic dodecahedron
  of $24$ tetrahedral tiles, and the associated spherical complex at infinity}
\end{figure}

As in Corollary \cref{cor:germclassification2}, there are three
$W(A_3^+)$-orbits of rank-2 germs, corresponding to the standard
convex polyhedral sets $P_{1}$, $P_{2}$, and $P_{3}$. The sets $P_{1}$
and $P_{3}$ are isometric by the order-2 isometry of $P_\emptyset$
given by the linear map fixing $\alpha_2$ and exchanging roots
$\alpha_1$ and $\alpha_3$. The set $P_{2}$ is isometric to neither, as
its recession cone has angle $\pi/2$, while the others have angle
$\pi/3$.

There are similarly three $W(A_3^+)$-orbits of rank-1 germs,
corresponding to $P_{1,2}$, $P_{1,3}$, and $P_{2,3}$. The sets
$P_{1,2}$ and $P_{2,3}$ are isometric, and look like beams whose
cross-section are equilateral triangles, corresponding to a tile in
the tessellation of $V_{1,2}$ which has type $A_2$. The set $P_{1,3}$
is a beam whose cross-section is a square, corresponding to a tile
in the tessellation of $V_{1,3}$ which has type $A_1\times A_1$.

\begin{figure}[htb]
  \centering
  \includegraphics[scale=0.7]{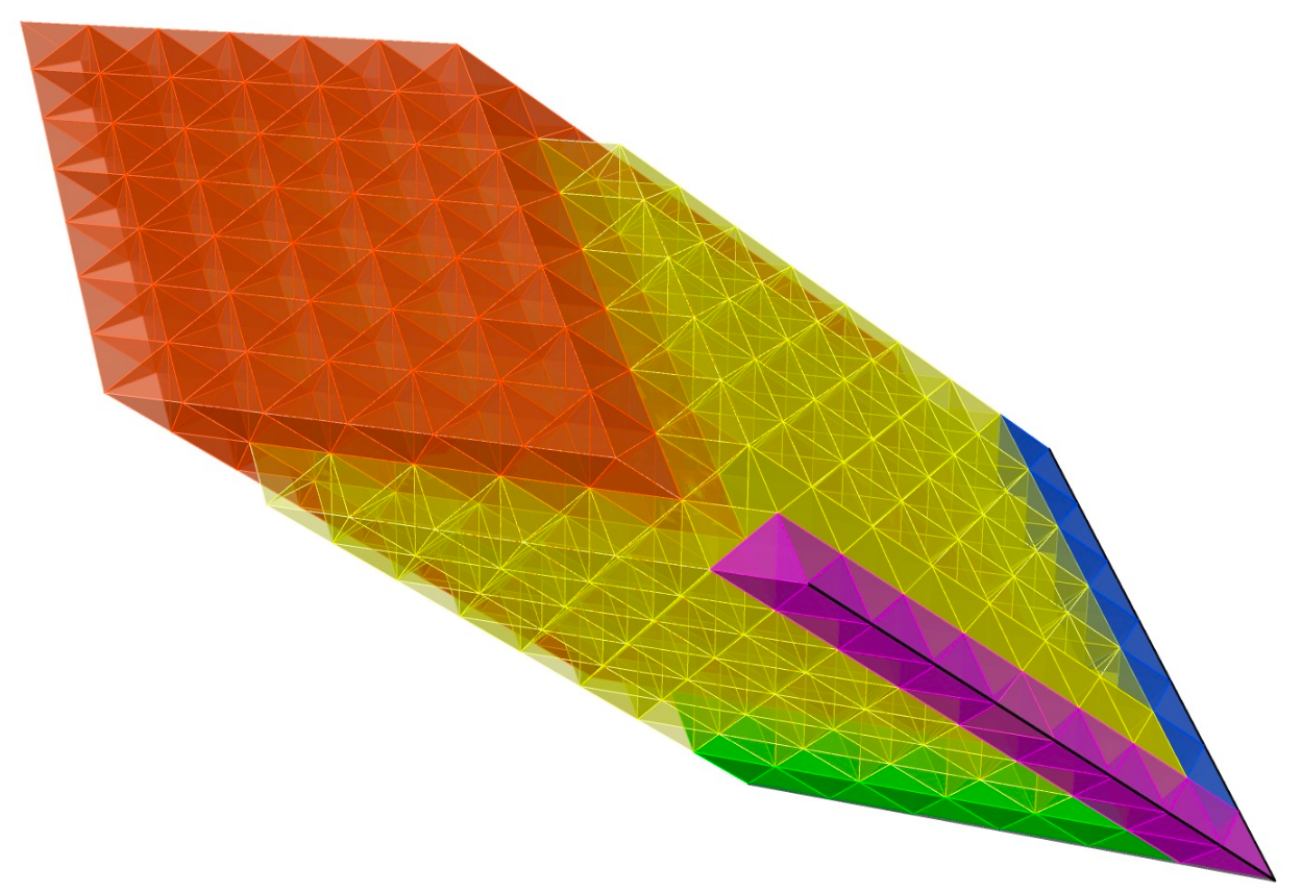}
\caption{\label{fig:A3alcoves} A rank-3 alcove in orange, with a
  translate in yellow. The rank-1 alcoves along the edges of the
  translate are colored purple, blue, and green.}
\end{figure}

\subsection{2D and 3D kagome tessellations}

Our results apply to tessellations that are not directly induced from
affine Weyl groups. Two such examples are the 2-dimensional kagome
tessellation, also known as the trihexagonal tiling, and its
3-dimensional generalization, also known as a quarter cubic
honeycomb. Both the 2- and 3-dimensional kagome tessellations may be
constructed in the following way. Consider a regular Euclidean simplex
$T$ in $\R^n$ where $n=2$ or $n=3$, so that $T$ is either a triangle
or tetrahedron. Let $H_0, H_1,\dotsc, H_n$ be the hyperplanes bounding
its faces. For each $H_i$, form a family of parallel hyperplanes
$\Omega_i$ regularly spaced at distance $\frac23$ times the height of
$T$. The collection of all such hyperplanes
$\Omega = \Omega_0\cup\dotsb \cup \Omega_n$ determines a tessellation
of $\R^n$, called the {\em kagome tessellation}.

\begin{figure}[htb]
  \centering
\centerline{\includegraphics[scale=0.4]{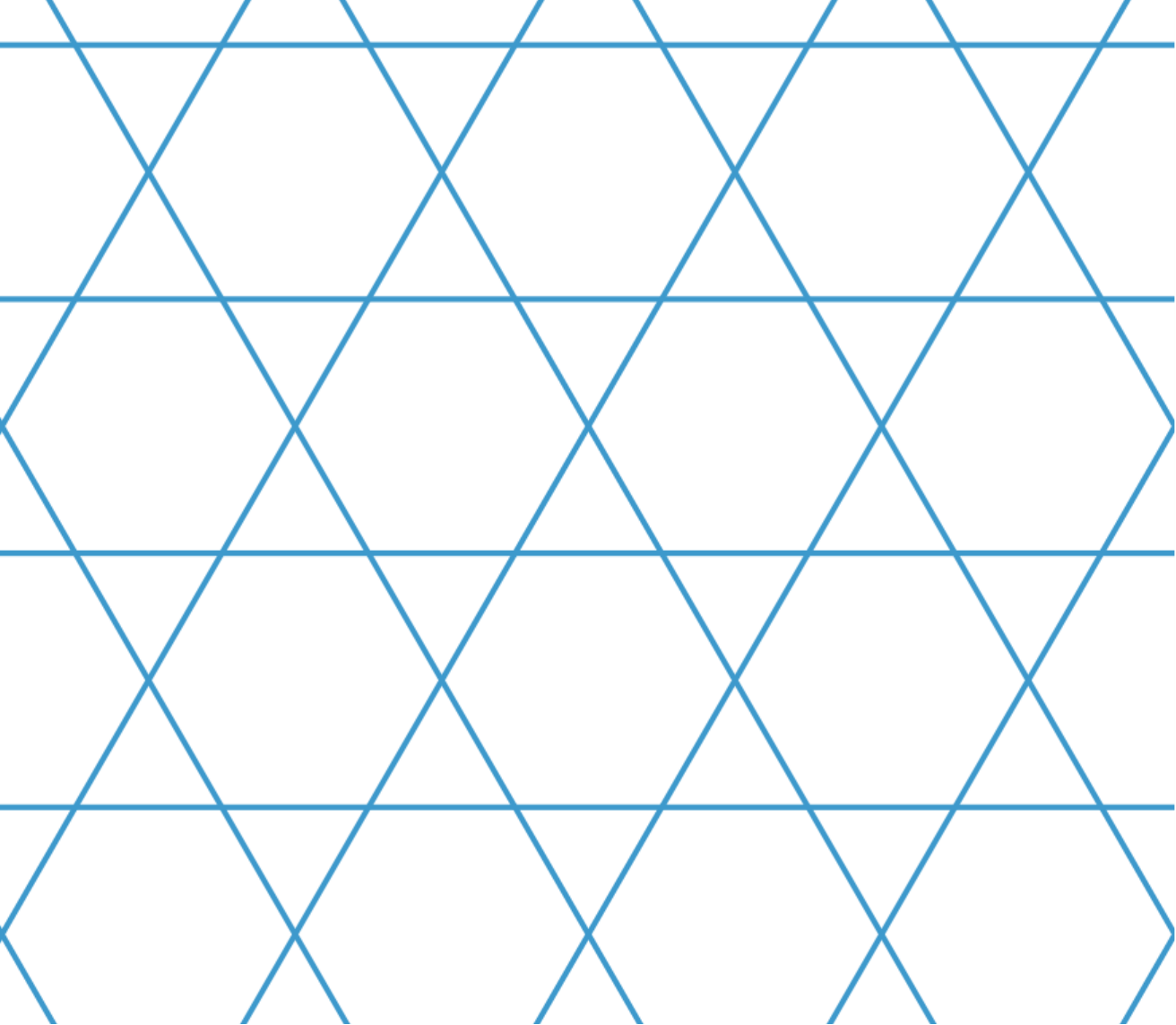}}
\caption{\label{fig:2dkagome} 2D kagome tessellation}
\end{figure}

\cref{fig:2dkagome} shows the 2D kagome tessellation $K_2$ which is made
from two kinds of tiles, equilateral triangles and regular
hexagons. The 2D kagome tessellation can be refined to the affine
$A_2^+$ tessellation by adding a line halfway between each pair of
adjacent parallel lines. In particular, the cell structure on the
sphere at infinity is the same as that of the $A_2^+$ tessellation,
and there is a similar classification of germs.

\cref{fig:3dkagome} shows the 3D kagome tessellation $K_3$ which
is made from two kinds of tiles, regular tetrahedra and truncated
tetrahedra. In contrast to the 2D kagome lattice, the 3D kagome
lattice is not related to the tessellation associated with the affine
Weyl group $W(A_3^+)$ in an obvious way. For example, the cell
structure on the sphere at infinity is significantly different, with
six triangular $2$-cells and six square $2$-cells formed by the four
circles corresponding to parallelism families of hyperplanes, which
intersect at twelve vertices.  See \cref{fig:3dkagome_sphere}.
Note that the square $2$-cells show that cells in the spherical
tessellation at infinity $L(\Delta)$ are not necessarily simplices.

\begin{figure}[htb]
  \centering
\centerline{\includegraphics[scale=0.2]{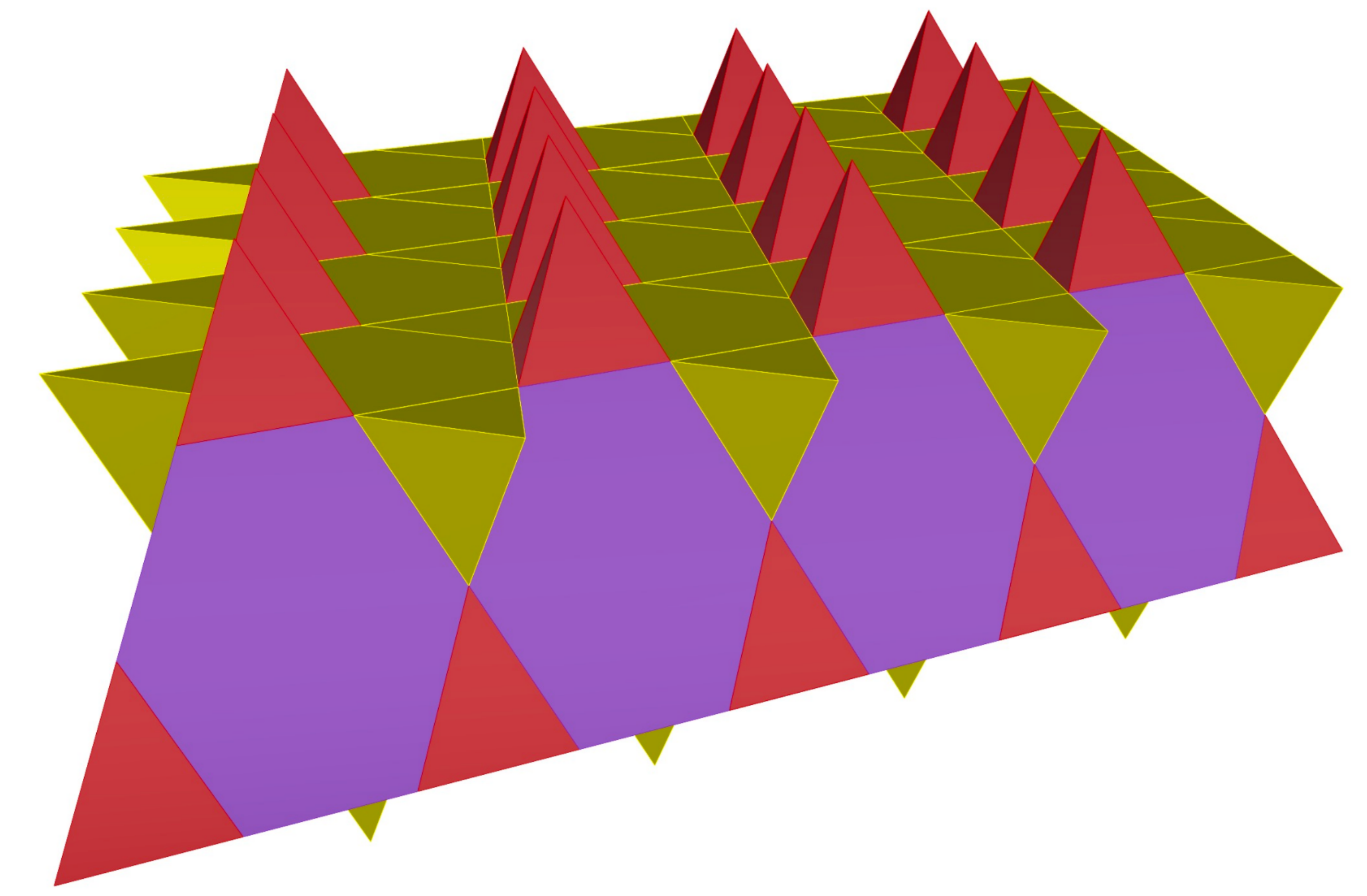}}
\caption{\label{fig:3dkagome} 3D kagome tessellation}
\end{figure}

\begin{figure}[htb]
  \centering
\centerline{\includegraphics[scale=0.5]{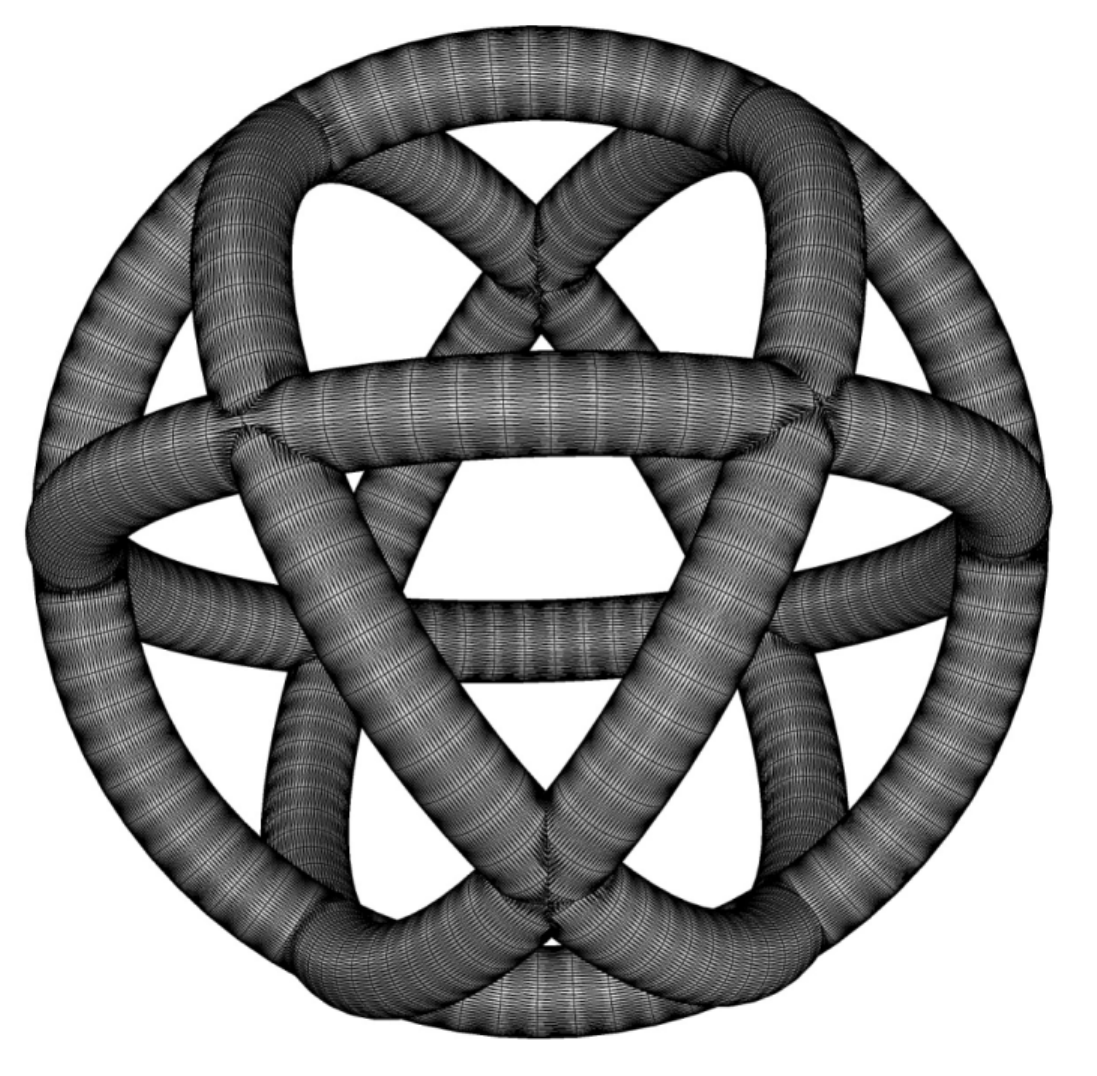}}
\caption{\label{fig:3dkagome_sphere} 3D kagome sphere at infinity}
\end{figure}

\end{document}